\documentclass[11pt]{article}
 
\usepackage[compress]{natbib}
\usepackage{amsmath,amssymb,amsthm}
\usepackage{color}
\usepackage[margin=1in]{geometry}
\usepackage{setspace}
\usepackage{amsthm, amsmath, amssymb, amsfonts}
\usepackage{subcaption}
\usepackage{bm}
\usepackage[dvipsnames]{xcolor}
\usepackage{hyperref}
\hypersetup{pdftex, colorlinks=true, citecolor=MidnightBlue, linkcolor=red!80!black, urlcolor=MidnightBlue}
\usepackage{caption}
\usepackage{url}            
\usepackage{booktabs}       
\usepackage{nicefrac}       
\usepackage{graphicx}
\usepackage{caption}
\usepackage{thm-restate}
\usepackage[export]{adjustbox}
 
\usepackage{microtype}
\usepackage{mathpazo} 
\usepackage[scaled]{helvet} 
\usepackage{courier} 
\normalfont
\usepackage[T1]{fontenc}

\usepackage{threeparttable}
\allowdisplaybreaks[4] 
\usepackage[shortlabels]{enumitem} 
 
\usepackage{tikz}
\usetikzlibrary{fit}
\tikzset{%
	highlight/.style={rectangle,blend mode = multiply,draw=blue!90!black,thick,rounded corners = 0.3 mm,inner sep=0.5pt}
}

\usepackage{tcolorbox}
\def\O#1{\text{\ding{\the\numexpr#1+171}}}

\newcommand{\bz}{\mathbf 0}

\newcommand{\R}{\mathbb R}

\DeclareMathOperator{\prox}{prox}

\def\N{{\mathbb{N}}}
\def\O{{\mathbb{O}}}
\def\P{{\mathbb{P}}}

\def\R{{\mathbb{R}}}

\def\cA{{\cal A}}

\def\cF{{\cal F}}

\def\cI{{\cal I}}

\def\cT{{\cal T}}

\def\cV{{\cal V}}

\def\a{\bm{a}}

\def\e{\bm{e}}

\def\g{\bm{g}}

\def\j{\bm{j}}

\def\m{\bm{m}}

\def\x{\bm{x}}
\def\y{\bm{y}}
\def\z{\bm{z}}
\def\bz{{\mathbf 0}}

\def\1{{\mathbf 1}}

\usepackage{thmtools}

\declaretheoremstyle[parent=section]{definitionwithend}

\declaretheorem[style=definitionwithend]{theorem}
\declaretheorem[style=definitionwithend]{proposition}
\declaretheorem[style=definitionwithend]{definition}

\declaretheorem[style=definitionwithend]{remark}

\declaretheorem[style=definitionwithend]{lemma}
\declaretheorem[style=definitionwithend]{fact}

\usepackage{mathtools}
\usepackage{pifont}
\usepackage[nameinlink]{cleveref}
\usepackage[titletoc]{appendix}
\declaretheorem[style=definitionwithend]{condition}

\crefname{assumption}{Assumption}{Assumptions}
\crefname{lemma}{Lemma}{Lemmas}
\crefname{theorem}{Theorem}{Theorems}
\crefname{corollary}{Corollary}{Corollaries}
\crefname{proposition}{Proposition}{Propositions}
\crefname{condition}{Condition}{Conditions}
\crefname{claim}{Claim}{Claims}
\crefname{figure}{Figure}{Figures}
\crefname{remark}{Remark}{Remarks}
\crefname{section}{Section}{Sections}
\crefname{example}{Example}{Examples}
\crefname{definition}{Definition}{Definitions}
\crefname{table}{Table}{Tables}
\crefname{equation}{}{}
\crefname{enumi}{}{}
\crefname{appendix}{Appendix}{Appendices}
\Crefname{appendix}{Appendix}{Appendices}
 
\usepackage{tikz}
\usetikzlibrary{arrows.meta,positioning}
 
\def\sA{{\mathsf{A}}}

\DeclareMathOperator{\rand}{rand}
\numberwithin{equation}{section}
 
\usepackage{mathrsfs}
\usetikzlibrary{shadows}
\newcommand{\fcoup}{\hat f}

\title{Optimal Oracle Complexity for Weakly Convex Optimization}
 
\date{\today}
\author{%
Jiajin Li \qquad Siyu Pan\\[0.75em]
Sauder School of Business,\\
The University of British Columbia\\[0.5em]
\small \texttt{\{jiajin.li,siyu.pan\}@sauder.ubc.ca}
}
\begin{document}
\maketitle
 
\begin{abstract}
We study the oracle complexity of finding $\epsilon$-stationary
points of $\rho$-weakly convex and $G$-Lipschitz functions,
where stationarity is measured by the gradient norm of the Moreau
envelope. We consider an exact first-order oracle that returns both
the function value and the full subdifferential at every
query point. For sufficiently small $\epsilon>0$, we prove that every
first-order algorithm, including randomized algorithms,
requires $\Omega({\rho G^2\Delta}/{\epsilon^4})$ oracle queries
to succeed with any fixed positive probability, where
$f(\bz)-\inf f\le\Delta$. This lower bound matches the best known deterministic
and stochastic first-order upper bounds, up to universal
constants, and establishes the optimal randomized and deterministic
first-order oracle complexity. The result separates smooth nonconvex and nonsmooth weakly
convex optimization: the additional $\epsilon^{-2}$ factor
is intrinsic to nonsmooth geometry rather than to stochastic
oracle noise.

 
\end{abstract}
 
\section{Introduction} 
\label{sec:introduction}
The oracle complexity of finding stationary points in nonconvex
optimization is well understood in the smooth setting.
Deterministic first-order methods achieve the optimal complexity
$\Theta(\epsilon^{-2})$
\citep{nesterov2013introductory,CarmonEtAl2021,carmon2020lower},
whereas stochastic first-order methods require
$\Theta(\epsilon^{-4})$ oracle queries when the stochastic gradients
have bounded variance
\citep{arjevani2023lower,ghadimi2013stochastic}.
 
The current complexity picture for weakly convex optimization is less conclusive. Weakly convex functions encompass smooth nonconvex and convex-composite objectives
\citep{vial1983strong,drusvyatskiy2019efficiency}, and provide a broad setting for nonsmooth nonconvex optimization.
For this class, stochastic subgradient methods achieve an oracle
complexity of $\mathcal O(\epsilon^{-4})$ for finding an
$\epsilon$-stationary point of the Moreau envelope
\citep{DavisDrusvyatskiy2019,davis2019proximally,davis2018stochastic}.
This rate is optimal for stochastic oracles, already for smooth
functions~\citep{arjevani2023lower}.
For deterministic methods, the same rate is attained by the
proximal descent method of \citet{liao2025proximal};
see also \citet{kong2024cost} for an analogous bound
under Goldstein stationarity.
This naturally raises the following question:
\begin{center}
    \bfseries Can deterministic first-order methods break the $\epsilon^{-4}$ barrier?
\end{center}
We show that the $\epsilon^{-4}$ dependence is already necessary
for deterministic first-order methods. This separates nonsmooth weakly
convex from smooth nonconvex optimization, where deterministic
methods need only $\Theta(\epsilon^{-2})$ queries: the additional
$\epsilon^{-2}$ factor is intrinsic to nonsmooth geometry rather
than to stochastic oracle noise. The separation parallels the
classical gap between smooth and nonsmooth convex optimization
\citep{NemirovskiYudin1983,nesterov2013introductory,bubeck2015convex},
but concerns stationarity rather than objective suboptimality. 
 
We further show that randomization does not help: even with exact
first-order oracles, randomized methods also require
$\Omega(\epsilon^{-4})$ queries. Hence, under weak convexity,
deterministic and randomized first-order methods share the same
optimal dimension-free complexity. In contrast, for general
Lipschitz functions and Goldstein stationarity, randomized methods
attain dimension-free rates
\citep{zhang2020complexity,tian2022finite,davis2022gradient},
whereas no deterministic method can
\citep{TianSo2024,jordan2023deterministic,kornowski2022complexity}.
 
Formally, we study the oracle complexity of the weakly convex
optimization problem
\begin{equation}
\label{eq:prob}
\min_{\x\in\mathbb R^d} f(\x),
\tag{P}
\end{equation}
where \(f:\mathbb R^d\to\mathbb R\) is globally \(G\)-Lipschitz,
\(\rho\)-weakly convex, and possibly nonsmooth. We assume that \(f(\bz)-\inf f\le\Delta\).
Let \(f_\lambda\) denote the Moreau envelope of \(f\), with
\(\lambda=\tfrac{1}{2\rho}\).
For \(\epsilon>0\), we call a point \(\x\in\mathbb R^d\)
\(\epsilon\)-stationary if
$
\|\nabla f_\lambda(\x)\|\le\epsilon.
$
We adopt a first-order oracle that returns both the function value
and the \emph{full subdifferential} at every query point.
Compared with the standard oracle that returns a single subgradient,
this stronger oracle ensures that our lower bound does not rely
on an unfavorable choice of subgradient by the oracle.
 
\subsection{Our results and techniques} 
We first state an informal version of our main lower bound. 
\begin{theorem}[Informal]
For all sufficiently small \(\epsilon>0\), the randomized first-order
oracle complexity of finding an \(\epsilon\)-stationary point
with probability at least \(1-\nu\) is
\[
\Omega\!\left(\frac{\rho G^2\Delta}{\epsilon^4}\right),
\]
with universal implied constants independent of \(\nu\in[0,1)\).
\end{theorem}
To prove this result, we generalize the classical
zero-chain construction
\citep{nesterov2013introductory,carmon2020lower}
in two directions.
First, we extend the zero-chain property to the full
subdifferential oracle. We construct a tilted max-affine
convex block $\phi:\R^N\to\R$ whose entire subdifferential
satisfies the following \emph{setwise zero-chain property} for every $n\in\{0,\ldots,N-1\}$: 
\[
 \x\in\operatorname{span}\{\bm e_1,\ldots,\bm e_n\}
 \quad\Longrightarrow\quad
 \partial \phi(\x)\subseteq
 \operatorname{span}\{\bm e_1,\ldots,\bm e_{n+1}\}.
\]
Thus even observing the entire subdifferential reveals
at most one new coordinate.
 
Second, we couple $M$ setwise zero-chain blocks
to construct a weakly convex hard instance.
In contrast to the smooth setting, where a single zero-chain
suffices~\citep{carmon2020lower}, this coupling is the main
technical difficulty.
The coupling permits simultaneous activation of
neighboring blocks, so the classical coordinatewise
zero-chain property is no longer preserved.
To track this simultaneous progress, we organize
coordinates across blocks into diagonal layers,
each containing at most a universal constant
number of coordinate directions.
We establish a generalized zero-chain property
with respect to the resulting filtration,
yielding an oracle horizon of order $MN$; see
Figure~\ref{fig:overlapping-relay}.
For sufficiently small $\epsilon$, we take
$N=\Theta(\tfrac{G^2}{\epsilon^2})$ and
$M=\Theta(\tfrac{\rho\Delta}{\epsilon^2})$, giving
$MN=\Theta(\tfrac{\rho G^2\Delta}{\epsilon^4})$.
 
\usetikzlibrary{arrows.meta}
\begin{figure}[t]
\centering
\begin{tikzpicture}[
    x=0.36cm,
    y=0.76cm,
    cell/.style={
      draw=black!45,
      line width=0.35pt,
      minimum width=0.36cm,
      minimum height=0.38cm,
      inner sep=0pt
    },
    revealed/.style={cell,fill=blue!18},
    frontier/.style={cell,fill=orange!50},
    hidden/.style={cell,fill=black!3},
    >=Latex
]
\def\BlockLength{12}
\def\RelayDelay{4}
\def\CurrentLevel{10}
\def\TotalDepth{24}

\pgfmathtruncatemacro{\NextLevel}{\CurrentLevel+1}
\foreach \m in {1,2,3} {
  \pgfmathtruncatemacro{\shift}{(\m-1)*\RelayDelay}
  \pgfmathtruncatemacro{\yy}{1-\m}

  \node[anchor=east]
    at (-0.65,\yy+0.25) {block \(\m\)};

  \foreach \j in {1,...,\BlockLength} {
    \pgfmathtruncatemacro{\lev}{\shift+\j-1}
    \ifnum\lev<\NextLevel
      \node[revealed] at (\lev+0.5,\yy+0.25) {};
    \else
      \ifnum\lev=\NextLevel
        \node[frontier] at (\lev+0.5,\yy+0.25) {};
      \else
        \node[hidden] at (\lev+0.5,\yy+0.25) {};
      \fi
    \fi
  }
}

\node at (10.8,-2.68) {\(\ddots\)};
\node[anchor=east] at (-0.65,-3.75) {block \(M\)};

\foreach \j in {1,...,\BlockLength} {
  \pgfmathtruncatemacro{\xx}{12+\j-1}
  \node[hidden] at (\xx+0.5,-3.75) {};
}

\draw[<->,line width=0.6pt]
  (0,1.05)--(\RelayDelay,1.05)
  node[midway,above] {\(\ell=\Theta(N)\)};

\draw[densely dotted,black!35]
  (0,0.84)--(0,-4.15);
\draw[densely dotted,black!35]
  (\RelayDelay,0.84)--(\RelayDelay,-4.15);

\draw[-{Latex[length=2mm]},blue!58!black,line width=0.8pt]
  (\RelayDelay-0.30,0.05)
  to[out=-36,in=145]
  (\RelayDelay+0.18,-0.72);

\draw[dashed,red!68!black,line width=0.75pt]
  (\NextLevel,0.72)--(\NextLevel,-2.35);
\node[anchor=south,text=red!65!black]
  at (\NextLevel,0.79) {diagonal layer};

\node[revealed,minimum width=0.32cm,minimum height=0.25cm]
  at (3.2,-4.72) {};
\node[anchor=west]
  at (3.55,-4.72) {revealed};

\node[frontier,minimum width=0.32cm,minimum height=0.25cm]
  at (9.1,-4.72) {};
\node[anchor=west]
  at (9.45,-4.72) {next layer};

\node[hidden,minimum width=0.32cm,minimum height=0.25cm]
  at (14.3,-4.72) {};
\node[anchor=west]
  at (14.65,-4.72) {not yet visible};

\draw[<->,line width=0.6pt]
  (0,-5.42)--(\TotalDepth,-5.42)
  node[midway,below=3pt] {total length \(\Theta(MN)\)};

\end{tikzpicture}
\caption{Schematic of the overlapping relay across blocks.}
\label{fig:overlapping-relay}
\end{figure}
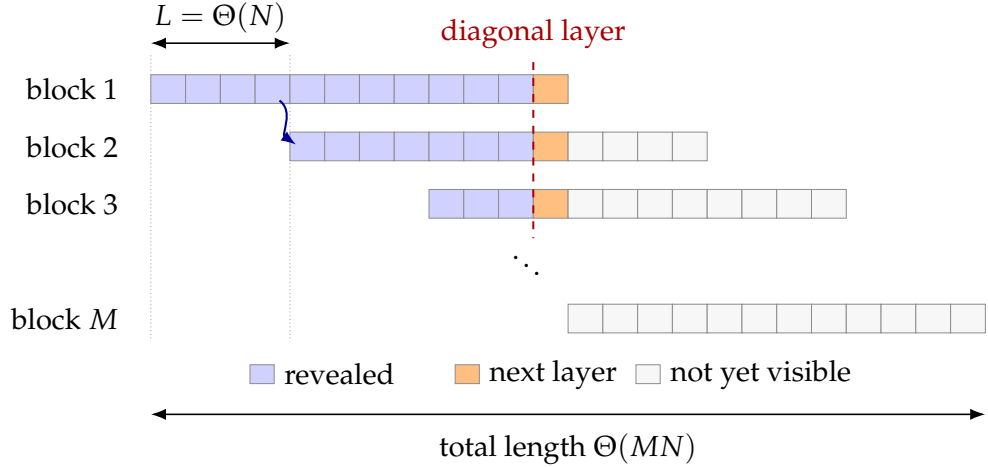
 
Finally, to handle arbitrary randomized first-order algorithms,
we combine the radial modification of \citet[Section~2.1]{DiakonikolasGuzman2020}
with a simplified adaptation of the random rotation
arguments of \citet{woodworth2016tight}
and \citet[Section~5.1]{carmon2020lower}. 
 
\paragraph{Notation.}
Bold lowercase letters denote vectors, with scalar coordinates
written without boldface, and bold uppercase letters denote matrices.
We use $\langle\cdot,\cdot\rangle$,
$\|\cdot\|$, and $\|\cdot\|_1$ for the Euclidean inner product,
Euclidean norm, and $\ell_1$ norm, respectively.
We write $\N:=\{1,2,\ldots\}$, $\N_0:=\N\cup\{0\}$,
$[n]:=\{1,\ldots,n\}$, and $[0]:=\varnothing$,
and adopt $\inf\varnothing:=+\infty$ and
$\operatorname{span}\varnothing:=\{\bz\}$.
The symbols $\bz$, $\bm e_i$ (or $\bm e_{mn}$ for block
coordinates), and $\mathbf I_d$ denote the zero vector,
standard basis vectors, and identity matrix, respectively.
Vector inequalities are componentwise; $\wedge$ and $\vee$
denote componentwise minimum and maximum, respectively, and
$\R^d_+$ and $\R^d_-$ denote the nonnegative and nonpositive orthants.
We write $[\x]_+:=(\max\{x_i,0\})_i$ and
$\operatorname{supp}(\x):=\{i\mid x_i\ne0\}$.
For a set $S$, $\operatorname{conv}S$ denotes its convex hull
and $\operatorname{dist}(\x,S):=\inf_{\y\in S}\|\x-\y\|$.
Finally, $\operatorname{Lip}(g)$ denotes the Lipschitz
constant of $g$, and $\prox_{\lambda f}(\x)$ denotes the
minimizer defining $f_\lambda(\x)$, which is unique for
$\rho$-weakly convex $f$ and $\lambda<\tfrac1\rho$.
 
\paragraph{Structure of the paper.} Section~\ref{sec:pre} introduces the function class,
stationarity measure, oracle model, and oracle complexity.
Section~\ref{sec:main} states the main result
and the properties required of the unscaled hard instance.
Section~\ref{sec:construction} constructs the unscaled
hard instance, and Section~\ref{sec:verification}
verifies its required properties.
Section~\ref{sec:proof} completes the proof through
the radial modification, random rotation, and rescaling.
 
\vspace{-2mm}
 
\section{Preliminaries}
\label{sec:pre}  
This section introduces the function class, stationarity
measure, and full subdifferential oracle. We then define the class of randomized first-order algorithms and the
associated oracle complexity.
\begin{definition}[Function class]
Suppose that \(G,\rho>0\) and \(\Delta\geq0\). We define \(\mathcal F(G,\rho,\Delta)\)
as the union, over dimensions \(d\in\N\), of the classes of functions
\(f:\R^d\to\R\) satisfying the following properties:
\begin{enumerate}[label=(\roman*)]
 \item (\(\rho\)-weak convexity) \(f\) is \(\rho\)-weakly convex,
 i.e., $\x\mapsto f(\x)+\tfrac\rho2\|\x\|^2$ is convex.
 \item ($G$-Lipschitz continuity) $f$ is $G$-Lipschitz, i.e., for all $\x,\y\in \R^d$,
 \[
 |f(\x)-f(\y)|\leq G\|\x-\y\|.
 \]
 \item (Initial gap) The initial point $\bz$ satisfies $f(\bz)-\inf_{\x\in\R^d} f(\x)\leq\Delta$.
\end{enumerate}
\end{definition}
We introduce the stationarity measure used throughout the paper. 
\begin{definition}[$\epsilon$-stationary point]
\label{def:stationary}
Let $\lambda = \tfrac{1}{2\rho}$, and define the Moreau envelope of $f$ by
\[
f_\lambda(\x)
:=
\min_{\z\in\R^d}
\left\{
f(\z)+\frac{1}{2\lambda}\|\x-\z\|^2
\right\}.
\]
Given $\epsilon>0$, a point $\x\in\R^d$ is an
$\epsilon$-stationary point of problem \eqref{eq:prob} if
$
\|\nabla f_\lambda(\x)\|\leq\epsilon.
$ 
\end{definition}
We consider the generalized first-order oracle, denoted by
$\mathbb O_f(\x)$. For finite-valued weakly convex functions on $\R^d$, the
Fr\'echet, limiting (Mordukhovich), and Clarke subdifferentials
coincide at every point, and their common value is a nonempty
compact convex set
\citep{rockafellar2009variational,clarke1990optimization,
mordukhovich2006variational,li2020understanding}; hence the
oracle below is well defined.
We call functions in $\mathcal F(G,\rho,\Delta)$ admissible.
\begin{definition}[{First-order oracle}]
For each query point \(\x\in\R^d\) and each admissible \(f\), the first-order oracle returns both the function value and the full subdifferential set, i.e.,
$\mathbb{O}_f(\x) := (f(\x),\partial f(\x))$. 
\end{definition}
   Let $\mathcal K_d$ denote the space of nonempty compact convex
   subsets of $\R^d$, equipped with the Borel $\sigma$-algebra
   induced by the Hausdorff metric, so that
   $\mathbb O_f(\x)\in\R\times\mathcal K_d$.
 
\begin{definition}[Randomized first-order algorithm]
\label{def:randomized-algorithm}
A randomized first-order algorithm $\sA$ is specified by
a private random seed
$\xi_{\sA}\sim\operatorname{Unif}[0,1]$, drawn independently
of $f$, and a sequence of Borel measurable maps
\[
\Gamma_t:
[0,1]\times(\R\times\mathcal K_d)^t
\longrightarrow \R^d,
\qquad t\geq1.
\]
An algorithm is understood as a family of such maps,
one for each dimension $d$; we suppress the dimension index.
For every admissible $f$, the algorithm starts from
$\sA^{(0)}[f]=\bz$ and generates its queries according to
\[
\sA^{(t)}[f]
=
\Gamma_t\left(
    \xi_{\sA},
    \bigl(
        \mathbb O_f(\sA^{(s)}[f])
    \bigr)_{s=0}^{t-1}
\right),
\qquad t\geq1.
\]
We denote the class of such algorithms by
$\mathcal A_{\rand}$.
\end{definition}
Notably, deterministic first-order algorithms form a subclass of
$\mathcal A_{\rand}$, obtained by choosing decision maps
$\Gamma_t$ that do not depend on the random seed $\xi_{\sA}$.
Consequently, our lower bounds for randomized algorithms
also apply to deterministic first-order algorithms.
 
Adapting the oracle complexity framework of
\citet{carmon2020lower,CarmonEtAl2021} to our setting,
we introduce the complexity measure used throughout the paper.
\begin{definition}[Oracle complexity]
\label{def:complexity-algorithm-class}
Fix a stationarity tolerance \(\epsilon>0\) and a confidence level
\(1-\nu\), where \(\nu\in[0,1)\). Let
\(\cA\subseteq\cA_{\rand}\) be any nonempty subclass of randomized
first-order algorithms, and let \(\cF\) be a nonempty admissible
function class. The oracle complexity of \(\cA\) on \(\cF\), denoted by
$\cT_{\epsilon,\nu}(\cA,\cF)$, is defined by
\[
    {\color{black}\inf_{\sA\in\cA}
    \sup_{f\in\cF}
    \inf}
    \left\{
        t\in\N_0
        \;\middle|\;
        \P_{\xi_{\sA}}
        \left(
            \sA^{(t)}[f]
            \text{ is an }\epsilon\text{-stationary point of }f
        \right)
        \geq 1-\nu
    \right\},
\]
where \(\P_{\xi_{\sA}}\) denotes probability with respect to the
internal seed \(\xi_{\sA}\) of the algorithm, as introduced in
Definition~\ref{def:randomized-algorithm}.
\end{definition}
The output $\sA^{(t)}[f]$ is formed after $t$ oracle replies
and need not itself be queried. In particular, $t=0$
allows returning the initial point $\bz$. 
\section{Main Results and Properties of the Hard Instance}
 \label{sec:main}
\subsection{Main results}
\begin{theorem}
\label{thm:main}
Let $G,\rho,\Delta>0$. Suppose that
$0<\epsilon\le\tfrac1{512}\min\{G,\sqrt{\rho\Delta}\}$ and
$\Delta\leq\tfrac{G^2}{\rho}\exp\left(\tfrac{G^2}{200\epsilon^2}\right)$.
Then, for every $\nu\in[0,1)$,
\[
\mathcal T_{\epsilon,\nu}
\bigl(\mathcal A_{\rand},\mathcal F(G,\rho,\Delta)\bigr)
\ge
2^{-29}\,\frac{\rho G^2\Delta}{\epsilon^4}.
\]
\end{theorem}
 
The lower bound allows any output that is a Borel measurable
function of the algorithm's private seed and oracle transcript. 
Moreover, this lower bound is matched, up to universal constants, by
the deterministic proximal descent method of
\citet[Theorem~2 and Corollary~4.2]{liao2025proximal}.
\citet{kong2024cost} give a related deterministic
complexity bound for Goldstein stationarity.
The same dependence on $\rho$, $G$, $\Delta$, and $\epsilon$
is attained by stochastic subgradient methods
\citep{DavisDrusvyatskiy2019,davis2018stochastic}.

\begin{remark}[Role of the weak convexity modulus]
\citet{kornowski2022oracle} and, independently,
\citet{tian2021hardness} showed that without control of
the weak convexity modulus, no first-order method can find
near-approximate stationary points of Lipschitz functions within a
dimension-free oracle budget. More precisely, \citet{tian2021hardness} showed that,
for every finite oracle horizon $T$, there exists a finite weak
convexity modulus $\rho(T)$ such that no deterministic first-order
method can guarantee finding a near-approximate stationary point within
$T$ oracle calls over the corresponding class of weakly convex
functions. This shows that controlling $\rho$ is essential for the
tractability of weakly convex optimization.
\end{remark}
 
\subsection{Properties of the hard instance}
We organize the construction around four properties
of an unscaled hard instance.
We call a pair of
integers $(M, N)$ valid
if $N\ge256$ and $2\le M\le2^{8N}$, and set $K=MN$.  
We consider a hard instance $\bar f:\R^K\to\R$, a positive
integer $T$ representing the oracle horizon, and
a nested coordinate filtration
\[
 \{\bz\}=\cV_{-1}\subseteq\cV_0
 \subseteq\cV_1\subseteq\cdots\subseteq\R^K.
\]
More precisely, for every integer $t\ge-1$,
there is an index set $\cI_t\subseteq[M]\times[N]$
such that
\[
 \cV_t
 =\operatorname{span}
 \{\bm e_{mn}\mid(m,n)\in\cI_t\},
 \qquad
 \cI_t\subseteq\cI_{t+1},
 \qquad
 \cI_{-1}=\varnothing.
\]
Define the coordinate projection $\mathbf{\Pi}_t$ by
\[
 (\mathbf{\Pi}_t\z)_{mn}
 :=
 \begin{cases}
 z_{mn}, & (m,n)\in\cI_t,\\
 0, & (m,n)\notin\cI_t,
 \end{cases}
\]
and set $\eta:=\tfrac1{64N(K+1)}$.
 
The following conditions describe the required
regularity, generalized zero-chain property, stationarity
bound, and initial gap.
\begin{condition}[Unscaled hard instance properties]
\label{cond:hard-instance}
The function $\bar f$, the coordinate filtration
$\{\cV_t\}_{t\ge-1}$, and the horizon $T$
satisfy the following properties.
\begin{enumerate}
 \item[\textnormal{(C1)}]
 The function $\bar f$ is finite-valued, bounded below,
 $2$-Lipschitz, and $\tfrac92$-weakly convex.
  \item[\textnormal{(C2)}]
For every integer $t\ge-1$, every $\z\in\R^K$ satisfying
 $|z_{mn}|<\eta$ for all $(m,n)\notin\cI_t$
 obeys
 \[
  \bar f(\z)=\bar f(\mathbf{\Pi}_{t+1}\z),
  \qquad
  \partial \bar f(\z)
  =\partial \bar f(\mathbf{\Pi}_{t+1}\z)
  \subseteq\cV_{t+1}.
 \]
 \item[\textnormal{(C3)}]
 The horizon satisfies $T\ge\tfrac{MN}{16}$,
 and every $\z\in\R^K$ satisfying
 $|z_{mn}|<\eta$ for all $(m,n)\notin\cI_{T-1}$
 obeys
 \[
  \left\|\nabla \bar f_{\tfrac19}(\z)\right\|
  >\frac1{10\sqrt N}.
 \]
  \item[\textnormal{(C4)}]
 The initial gap satisfies
 $
  \bar f(\bz)-\inf \bar f\le\frac{5M}{N}.
 $
 \end{enumerate}
 \end{condition}
 
\begin{remark}[Role of the tolerance $\eta$]
Conditions \textnormal{(C1)--(C4)} suffice for the
deterministic lower bound through a classical resisting-rotation and rescaling argument; see the first version of this
work~\citep{li2026optimalv1}.
The resisting rotation makes the unrevealed coordinates
of the queries and output exactly zero.
For randomized algorithms, we use a random rotation drawn
independently of the algorithm, under which the unrevealed
coordinates of bounded queries are small with high probability.
The tolerance $\eta$ ensures that \textnormal{(C2)} and
\textnormal{(C3)} remain applicable, preserving the
generalized zero-chain property and stationarity bound.
\end{remark}
 
To handle randomized queries of arbitrary norm, we combine the
radial modification of \citet[Section~2.1, Eq.~(3)]{DiakonikolasGuzman2020} with a
simplified adaptation of the random rotation arguments of
\citet{woodworth2016tight} and
\citet[Section~5.1]{carmon2020lower} to the full subdifferential oracle. 
 
For $d\ge K$ and
$\mathbf{U}\in\R^{d\times K}$ satisfying
$\mathbf{U}^\top\mathbf{U}=\mathbf{I}_K$, define
\begin{equation}
 h_{\mathbf{U}}(\bm z)
 :=
 \max\left\{
 \frac12\bar f\left(\frac23\mathbf{U}^\top\bm z\right),
 \frac12\bar f(\bz)+\|\bm z\|-R
 \right\},
 \label{eq:radial-instance}
\end{equation}
where $R=R(M,N)>0$ is specified in Lemma~\ref{lem:radial-stationarity}. 
For $\|\bm z\|>3R$, the radial term strictly
dominates, so the full oracle is independent
of the hidden directions $\mathbf{U}$.
 
In Section~\ref{sec:proof}, we use further properties
of our construction beyond \textnormal{(C1)--(C4)}
to show that $h_{\mathbf U}$ satisfies an analogue
of the generalized zero-chain property in
\textnormal{(C2)} and a stationarity bound analogous
to \textnormal{(C3)}.
We then combine these properties with a simplified
random rotation  and rescaling  argument to establish
the desired lower bound.

\section{The Construction of Our Hard Instance}
\label{sec:construction}
Throughout this section and the next, fix a valid
pair $(M,N)$.
We identify $\R^K$ with $M$ ordered blocks of $N$ coordinates
and write $\z=(\z^1,\ldots,\z^M)$ with $\z^m=(z_{m1},\ldots,z_{mN})\in\R^N$. 
Here, $m\in[M]$ indexes the blocks, while $n\in[N]$ indexes
the coordinates within each block.
 
The construction proceeds in two stages. We first introduce a tilted max-affine convex block whose full subdifferential satisfies a setwise zero-chain property. We then couple $M$
such blocks so that activating each subsequent block requires sufficient progress in the preceding block, while allowing
progress in neighboring blocks to overlap. Since the Lipschitz constant of the coupled function grows with $M$,
we further construct the least $2$-Lipschitz majorant of the
coupled function. The resulting construction satisfies a generalized
zero-chain property with respect to a diagonal coordinate filtration,
yielding an oracle horizon of order $MN$.
 
\subsection{A tilted max-affine block}
For $m\in[M]$ and $n\in[N]$, define
\begin{equation}
 c_{mn}
 :=1+\frac{K-(m-1)N-n+1}{32(K+1)}.
 \label{newrand:fan}
\end{equation}
These intercepts lie in $(1,\tfrac{33}{32})$ and decrease strictly
in the lexicographic order of $(m,n)$.
For each $m\in[M]$, define the tilted max-affine function
\[
 \phi_m(\bm{u})
 :=
 \max\left\{
   0,\,
   \max_{n\in[N]}
   \left(c_{mn}-\frac N8 u_n\right)
 \right\},
 \qquad \bm{u}\in\R^N.
\]
This block adapts the max-affine constructions used for
nonsmooth convex lower bounds
\citep{WoodworthSrebro2017,DiakonikolasGuzman2020}.
When the unrevealed coordinates are zero, the strict ordering
of the intercepts ensures that, among the unrevealed affine pieces,
only the first can be active. This yields the following \emph{setwise
zero-chain property}.
\begin{lemma}
\label{lem:tilted-block-locality}
For every {$m\in[M]$ and} \(n\in\{0,\ldots,N-1\}\), we have 
\[
 \bm{u}\in\operatorname{span}\{\e_1,\ldots,\e_n\}
 \quad\Longrightarrow\quad
 \partial\phi_{m}(\bm{u})
 \subseteq
 \operatorname{span}\{\e_1,\ldots,
\e_{n+1}\}.
\]
\end{lemma}
\begin{proof}
For every $j\in[N]$ with $j>n+1$, we have $u_j=u_{n+1}=0$ and
$c_{mj}<c_{m,n+1}$, so the $j$-th affine piece is inactive.
The subdifferential is the convex hull of the active gradients,
all of which belong to
$\operatorname{span}\{\e_1,\ldots,\e_{n+1}\}$.
\end{proof}
 
The tilted max-affine block satisfies the setwise zero-chain
property, with a chain of length $N$. Taking $M$ independent
copies does not by itself yield an $MN$-step lower bound,
since a single oracle query can reveal a new coordinate in every
block at once, so that $N$ queries suffice to traverse all blocks.
We therefore couple consecutive
blocks so that block $m$ can become active only after
sufficient progress has been made in block $m-1$.
The coupled construction admits a diagonal filtration
of constant width and an oracle horizon of order $MN$,
providing the additional factor $M$ in the lower bound.
 
\subsection{Coupling the tilted blocks}
\label{subsec:block-coupling} 
To couple consecutive blocks, we assign each block a scalar
measure of progress. Define the continuously differentiable
function $q:\R\to[0,1]$ by
\[
 q(s)
 :=
 \begin{cases}
  0, & s\le0,\\
  3s^2-2s^3, & 0<s<1,\\
  1, & s\ge1.
 \end{cases}
\]
Set $\theta:=\tfrac{1}{64N}$ and define the progress function
$\psi:\R^N\to[0,1]$ by 
\[
 \psi(\bm{u})
 :=
 \frac1N\sum_{n=1}^N
 q\left(\frac N4u_n-\theta\right).
\] 
The value $\psi(\z^m)$ records the progress within block $m$.
The shift $\theta$ ensures that an unrevealed coordinate
satisfying $|z_{mn}|\le\tfrac{4\theta}{N}$ contributes zero both
to $\psi(\z^m)$ and to the $n$-th component of
$\nabla\psi(\z^m)$.
 
We use this progress function to couple consecutive blocks.
For $m\in[M]$, define
\begin{equation}
 p_m(\z)
 :=h\bigl(\psi(\z^m)\bigr),
 \qquad
 h(s):=q\left(\frac{4s-1}{3}\right),
 \label{newrand:gate}
\end{equation}
and set $p_0(\z):=1$.
Thus $p_m$ depends on $\z$ only through $\z^m$; it vanishes whenever
$\psi(\z^m)\le\tfrac14$ and equals one whenever $\psi(\z^m)=1$.
 
We then define the block scores
\[
 \sigma_m(\z)
 :=\phi_m(\z^m)+p_{m-1}(\z)-\frac54,
 \qquad m\in[M],
\]
and write
$\boldsymbol{\sigma}(\z):=(\sigma_1(\z),\ldots,\sigma_M(\z))$.
Here, $p_{m-1}(\z)$ transmits progress from block $m-1$
to block $m$, while $p_0(\z)=1$ initializes the first block.
 
We combine the block scores and progress functions to define
the coupled function
\begin{equation}
 \label{eq:hard}
 \fcoup(\z)
 :=
 \frac2N
 \left(
   \|[\boldsymbol{\sigma}(\z)]_+\|
   -
   2\sum_{m=1}^M\psi(\z^m)
 \right).
\end{equation}
The first term aggregates the positive block scores. Through
$p_m$, progress in block $m$ raises $\sigma_{m+1}$; the second
term compensates for this increase, so that $\fcoup$ is
nonincreasing in every coordinate.
 
\begin{remark}[Overlapping activation across blocks]
\label{rem:block-relay}
The score $\sigma_m$ implements a relay between consecutive
blocks.
Consider an unstarted block $m\ge2$, so that
$\z^m=\bz$. Since $\phi_m(\bz)=c_{m1}\in(1,\tfrac{33}{32})$,
we have
\[
 \begin{aligned}
 \psi(\z^{m-1})\le\tfrac14
 &\quad\Longrightarrow\quad \sigma_m(\z)<0,\\
 \psi(\z^{m-1})\ge\tfrac12
 &\quad\Longrightarrow\quad \sigma_m(\z)>0.
 \end{aligned}
\]
Thus, the next block can become active before the
preceding block is complete.
Since $0\le q\le1$ and $q(-\theta)=0$, we have
$\psi(\z^{m-1})\le\tfrac{|\operatorname{supp}(\z^{m-1})|}{N}$.
Hence, an unstarted block cannot become active unless
more than $\lfloor \tfrac{N}{4}\rfloor$ coordinates in the
preceding block are nonzero. This threshold determines the
offset $\ell=\lfloor\tfrac N4\rfloor$ of the diagonal filtration
in Section~\ref{sec:locality}.
\end{remark}
 
Since the Lipschitz constant of $\fcoup$ may grow with $M$,
we further construct the least $2$-Lipschitz majorant of $\fcoup$:
\begin{equation}
 \bar f(\z)
 :=
 \sup_{\y\in\R^K}
 \left\{\fcoup(\y)-2\|\z-\y\|\right\}.
 \label{newrand:majorant}
\end{equation}
Taking $\y=\z$ gives $\bar f\ge\fcoup$, and
Proposition~\ref{prop:regularity} shows that $\bar f$ is finite
and $2$-Lipschitz. Equivalently, $-\bar f$ is the Pasch--Hausdorff envelope of $-\fcoup$;
see \citet[Example~9.11]{rockafellar2009variational}. 
 
   \section{Verification of the Hard Instance Properties}
\label{sec:verification}
\subsection{Lipschitz continuity and weak convexity}
\label{subsec:regularity} 
We verify \textnormal{(C1)}, namely that $\bar f$ is finite-valued,
bounded below, $2$-Lipschitz, and $\tfrac92$-weakly convex.
We begin with a variational representation of $\fcoup$.
Let
$\mathcal W:=\{\bm w\in\R_+^M\mid\|\bm w\|\le1\}$,
and adopt the convention $w_{M+1}:=0$.
\begin{lemma}[Variational representation]
\label{lem:max-representation}
For every $\z\in\R^K$,
\[
 \fcoup(\z)=\max_{\bm w\in\mathcal  W}F_{\bm w}(\z),
\]
where
\begin{equation}
 F_{\bm w}(\z)
 :=
 \frac2N
 \Biggl[
 w_1+\sum_{m=1}^M
 \Bigl(
 w_m\bigl(\phi_m(\z^m)-\tfrac54\bigr)
 +w_{m+1}p_m(\z)-2\psi(\z^m)
 \Bigr)
 \Biggr].
 \label{eq:Fw}
\end{equation}
\end{lemma}
\begin{proof}
For every $\a\in\R^M$, we have 
$
 \|[\a]_+\|
 =
 \max_{\bm w\in\mathcal W}\langle\bm w,\a\rangle.
$
Applying this identity to $\a=\boldsymbol{\sigma}(\z)$
gives
\[
 \fcoup(\z)
 =
 \frac2N
 \max_{\bm w\in\mathcal W}
 \left(
 \sum_{m=1}^M w_m\sigma_m(\z)
 -2\sum_{m=1}^M\psi(\z^m)
 \right).
\]
Since $p_0(\z)=1$ and $w_{M+1}=0$,
\[
 \sum_{m=1}^M w_m p_{m-1}(\z)
 =
 w_1+\sum_{m=1}^M w_{m+1}p_m(\z).
\]
Substituting
$\sigma_m(\z)=\phi_m(\z^m)+p_{m-1}(\z)-\tfrac54$
and collecting terms yields \eqref{eq:Fw}.
\end{proof}
In the variational representation, each tilted block
$\phi_m$ is convex, and all negative curvature comes
from the smooth terms involving $\psi$ and $p_m$.
To control this curvature, we first record estimates on $\psi$.
 
\begin{lemma}
\label{lem:progress-regularity}
For every $\bm u\in\R^N$, we have
\[
 0\le\psi(\bm u)\le1,
 \qquad
 \|\nabla\psi(\bm u)\|\le\frac{3\sqrt N}{8},
 \qquad
 \operatorname{Lip}(\nabla\psi)\le\frac{3N}{8}.
\]
\end{lemma}
 
\begin{proof}
On $0<s<1$, we have $q'(s)=6s(1-s)$ and
$q''(s)=6-12s$. The one-sided derivatives agree at
both junction points, so $q\in C^1(\R)$. Moreover, we have $ 0\le q\le1$, $ 0\le q'\le\frac32$ and $ \operatorname{Lip}(q')\le6$.
Since $0\le q\le1$, the definition of $\psi$ gives
$0\le\psi\le1$. By the chain rule,
\[
 \nabla\psi(\bm u)
 =
 \frac14
 \left(
 q'\left(\frac N4u_1-\theta\right),
 \ldots,
 q'\left(\frac N4u_N-\theta\right)
 \right).
\]
Since $0\le q'\le\tfrac32$,
\[
 \|\nabla\psi(\bm u)\|
 =
 \frac14
 \left(
 \sum_{n=1}^N
 \left|q'\left(\frac N4u_n-\theta\right)\right|^2
 \right)^{1/2}
 \le
 \frac14\sqrt{N\left(\frac32\right)^2}
 =
 \frac{3\sqrt N}{8}.
\]
Furthermore, for every $\bm u,\bm v\in\R^N$,
\[
 \|\nabla\psi(\bm u)-\nabla\psi(\bm v)\|^2
 \le
 \sum_{n=1}^N
 \left(\frac{3N}{8}|u_n-v_n|\right)^2
 =
 \frac{9N^2}{64}\|\bm u-\bm v\|^2.
\]
This proves the claimed estimates.
\end{proof}
We next establish the weak convexity of the coupled
function $\fcoup$.
 
\begin{lemma}[Weak convexity of the coupled function]
\label{lem:coupled-curvature}
The function $\fcoup$ is $\tfrac92$-weakly convex.
\end{lemma}
\begin{proof}
Since $h'(s)=\tfrac43q'(\tfrac{4s-1}{3})$, the bounds
$0\le q'\le\tfrac32$ and $\operatorname{Lip}(q')\le6$ give
$0\le h'\le2$ and $\operatorname{Lip}(h')\le\tfrac{32}{3}$.
 
By Lemma~\ref{lem:max-representation}, it suffices to show that
each $F_{\bm w}$ is $\tfrac92$-weakly convex. Since each $\phi_m$
is convex and the smooth terms in \eqref{eq:Fw} depend on disjoint
blocks, it suffices to show that, for every $\omega\in[0,1]$,
the function
$ G_\omega(\bm u)
 :=\omega h(\psi(\bm u))-2\psi(\bm u)
$ 
is $\tfrac{9N}{4}$-weakly convex; the prefactor $\tfrac2N$ in
\eqref{eq:Fw} then gives the modulus $\tfrac92$.
 
By the chain rule,
$
 \nabla G_\omega(\bm u)
 =
 \bigl(\omega h'(\psi(\bm u))-2\bigr)\nabla\psi(\bm u).
$ 
Since $0\le\omega h'\le2$, for any $\bm u,\bm v\in\R^N$,
\begin{align*}
 &\, \|\nabla G_\omega(\bm u)-\nabla G_\omega(\bm v)\| \\
 \le& \,
 2\|\nabla\psi(\bm u)-\nabla\psi(\bm v)\|+
 \operatorname{Lip}(h')\,|\psi(\bm u)-\psi(\bm v)|
 \,\|\nabla\psi(\bm v)\|.
\end{align*}
Since $0\le\omega h'\le2$, $\omega\le1$, and
$\operatorname{Lip}(\psi)\le\tfrac{3\sqrt N}{8}$ by
Lemma~\ref{lem:progress-regularity}, we obtain
\[
 \operatorname{Lip}(\nabla G_\omega)
 \le
 2\cdot\frac{3N}{8}
 +\frac{32}{3}\left(\frac{3\sqrt N}{8}\right)^2
 =\frac{9N}{4}.
\]
A continuously differentiable function with an
$L$-Lipschitz gradient is $L$-weakly convex.
Thus $G_\omega$ is $\tfrac{9N}{4}$-weakly convex, which
completes the proof.
\end{proof}
 
We next show that the majorant $\bar f$
is globally Lipschitz and retains the weak-convexity modulus of $\fcoup$.
\begin{proposition}
\label{prop:regularity}
The function $\bar f$ is finite-valued, bounded below,
$2$-Lipschitz, and $\tfrac92$-weakly convex.
\end{proposition}
\begin{proof}
We first show that $\bar f$ is bounded below.
Since $\|[\boldsymbol\sigma(\z)]_+\|\ge0$ and
$0\le\psi\le1$,
\begin{equation}
 \fcoup(\z)
 \ge-\frac4N\sum_{m=1}^M\psi(\z^m)
 \ge-\frac{4M}{N}.
 \label{eq:lower_f}
\end{equation}
Taking $\y=\z$ in \eqref{newrand:majorant} gives
$
 \bar f(\z)\ge \fcoup(\z)\ge-\tfrac{4M}{N}.
$ 
 
To prove finiteness, note that $c_{mn}<2$ implies
$\phi_m(\z^m)\le2+\tfrac{N}{8}\|\z^m\|$.
Since $[\sigma_m(\z)]_+\le\phi_m(\z^m)$ and $\psi\ge0$,
\begin{align*}
 \fcoup(\z)
 &\le\frac2N\|[\boldsymbol\sigma(\z)]_+\|\le\frac2N
 \left(\sum_{m=1}^M
 \left(2+\frac N8\|\z^m\|\right)^2\right)^{1/2}\le\frac{4\sqrt M}{N}+\frac14\|\z\|,
\end{align*}
where the last inequality follows from the triangle inequality
in $\R^M$.
 
Now fix $\z\in\R^K$. For every $\y\in\R^K$,
\begin{equation}
\label{eq:f0_coer}
\begin{aligned}
 \fcoup(\y)-2\|\z-\y\|
 &\le\frac{4\sqrt M}{N}
       +\frac14\|\y\|-2\|\z-\y\|\\
 &\le\frac{4\sqrt M}{N}
       +\frac14\|\z\|-\frac74\|\z-\y\|\le\frac{4\sqrt M}{N}+\frac14\|\z\|.
\end{aligned}
\end{equation}
Taking the supremum over $\y$ gives
\[
 \bar f(\z)
 \le\frac{4\sqrt M}{N}+\frac14\|\z\|
 <+\infty.
\]
 
For each fixed $\y$, the function
$\z\mapsto \fcoup(\y)-2\|\z-\y\|$ is $2$-Lipschitz.
Since $\bar f$ is finite-valued, taking the supremum over
$\y$ preserves this bound.
Finally, substituting $\y=\z+\bm v$ gives
$\bar f(\z)=\sup_{\bm v\in\R^K}
\{\fcoup(\z+\bm v)-2\|\bm v\|\}$.
For each fixed $\bm v$, the function inside the supremum
is $\tfrac92$-weakly convex.
Taking the supremum preserves this modulus, so
$\bar f$ is $\tfrac92$-weakly convex.
\end{proof}

\subsection{A robust generalized zero-chain property} 
\label{sec:locality} 
We next verify \textnormal{(C2)}: we characterize the full
subdifferential of $\bar f$ and establish its generalized
zero-chain property.
 
Let $\ell:=\lfloor \tfrac{N}{4}\rfloor$ and
fix the horizon $T:=(M-1)\ell$.
We assign each coordinate direction $\e_{mn}$ the level
$
 \operatorname{lev}(m,n):=(m-1)\ell+n-1.
$
Within each block, the level increases by one with each
coordinate. The first coordinate of block $m$ has level
$(m-1)\ell$.
For every integer $t\ge-1$, define
\begin{equation}
\label{newrand:filtration}
\begin{aligned}
 \cI_t
 &:=
 \left\{(m,n)\in[M]\times[N]\mid
 \operatorname{lev}(m,n)\le t\right\},\\
 \cV_t
 &:=
 \operatorname{span}
 \left\{\e_{mn}\mid(m,n)\in \cI_t\right\}.
\end{aligned}
\end{equation}
In particular, $\cI_{-1}=\varnothing$ and $\cV_{-1}=\{\bz\}$.
The nested subspaces $\{\cV_t\}_{t\ge-1}$ form the
diagonal filtration, and we call $\cI_{t+1}\setminus\cI_t$,
the set of coordinates of level $t+1$, a layer.
Since the first coordinates of consecutive blocks are $\ell$ levels
apart, each layer meets at most $\lceil \tfrac{N}{\ell}\rceil\le5$ blocks and
hence contains at most five coordinate directions.
 
 
We first compute the full subdifferential of $\fcoup$.
For $\bm w\in\mathcal W$ and $m\in[M]$, define
$
 a_m(\bm w,\z)
 :=
 2-w_{m+1}h'(\psi(\z^m))
 \in[0,2].
$
\begin{lemma}
\label{lem:full-subdiff}
Let $
 \mathcal W_\star(\z)
 :=
 \operatorname*{argmax}_{\bm w\in\mathcal W}
 \langle\bm w,\boldsymbol\sigma(\z)\rangle.
$
Then
\begin{equation}\label{eq:full-subdiff}
     \partial \fcoup(\z)
 =
 \operatorname{conv}
 \bigcup_{\bm w\in\mathcal W_\star(\z)}
 \mathcal G(\bm w,\z),
\end{equation}
where $\mathcal G(\bm w,\z)$ is the nonempty compact
convex set of vectors $\g=(\g^1,\ldots,\g^M)$ satisfying, for all $m \in [M]$,
\begin{equation*}
\label{eq:block-subgradient}
 \g^m
 =
 \frac2N
 \left(
 w_m\bm u^m
 -
 a_m(\bm w,\z)\nabla\psi(\z^m)
 \right),
 \qquad
 \bm u^m\in\partial\phi_m(\z^m).
\end{equation*}
Moreover, every $\bm w\in\mathcal W_\star(\z)$ satisfies $w_m=0$
whenever $\sigma_m(\z)<0$, and
$\mathcal W_\star(\z)=\left\{\tfrac{[\boldsymbol\sigma(\z)]_+}{\|[\boldsymbol\sigma(\z)]_+\|}\right\}$
whenever $[\boldsymbol\sigma(\z)]_+\ne\bz$.
\end{lemma}
\begin{proof}
For fixed $\bm w$, \eqref{eq:Fw} and the
subdifferential summation rule
\citep[Corollary~10.9]{rockafellar2009variational} give
$\partial F_{\bm w}(\z)=\mathcal G(\bm w,\z)$.
By Lemma~\ref{lem:max-representation},
$
 \fcoup(\z)=\max_{\bm w\in\mathcal W}F_{\bm w}(\z),
$ and the maximizers
at $\z$ are exactly $\mathcal W_\star(\z)$.
By the proof of Lemma~\ref{lem:coupled-curvature}, each
$F_{\bm w}+\tfrac94\|\cdot\|^2$ is convex; moreover, $\mathcal W$ is
compact and $\bm w\mapsto F_{\bm w}(\z)$ is continuous.
The convex maximum rule
\citep[Theorem~D.4.4.2]{hiriart2004fundamentals}, followed by
subtracting the common gradient $\tfrac92\z$, gives
\eqref{eq:full-subdiff}.
 
Moreover, fix $\bm w\in\mathcal W_\star(\z)$.
Since $\bm w\ge\bz$ and $\|\bm w\|\le1$,
\[
 \langle\bm w,\boldsymbol\sigma(\z)\rangle
 \le\langle\bm w,[\boldsymbol\sigma(\z)]_+\rangle
 \le\|[\boldsymbol\sigma(\z)]_+\|,
\]
and the maximum over $\mathcal W$ equals $\|[\boldsymbol\sigma(\z)]_+\|$.
Hence both inequalities hold with equality.
Equality in the first gives $w_m=0$ whenever $\sigma_m(\z)<0$.
If $[\boldsymbol\sigma(\z)]_+\ne\bz$, equality in the second gives
$\bm w =\tfrac{[\boldsymbol\sigma(\z)]_+}{\|[\boldsymbol\sigma(\z)]_+\|}$. 
\end{proof}

We now use this characterization to describe the
full subdifferential of $\bar f$. 
\begin{lemma}
\label{newrand:sources}
For each $\z\in\R^K$, the set
$\mathcal S(\z)
 :=\operatorname{argmax}_{\y\in\R^K}\{\fcoup(\y)-2\|\z-\y\|\}$
is nonempty and compact. Moreover,
\begin{equation}
\label{newrand:source-formula}
 \partial \bar f(\z)
 =
 \operatorname{conv}
 \bigcup_{\y\in\mathcal S(\z)}
 \partial \fcoup(\y).
\end{equation}
\end{lemma}
\begin{proof}
By \eqref{eq:f0_coer}, the function
$\y\mapsto 2\|\z-\y\|-\fcoup(\y)$ is continuous
and coercive, locally uniformly in $\z$.
Hence $\mathcal S(\z)$ is nonempty and compact.
 
Writing $\y=\z+\bm v$, the coercivity above shows that, locally
in $\z$, there exists $r>0$ such that
$ \bar f(\z)
 =\max_{\|\bm v\|\le r}
 \bigl\{\fcoup(\z+\bm v)-2\|\bm v\|\bigr\}$. 
By Lemma~\ref{lem:coupled-curvature}, for every fixed $\bm v$, the
function inside is $\tfrac92$-weakly convex. The convex maximum rule
\citep[Theorem~D.4.4.2]{hiriart2004fundamentals} therefore gives
$
 \partial\bar f(\z)
 =\operatorname{conv}
 \bigcup_{\bm v}\partial\fcoup(\z+\bm v),
$ 
where the union is over the maximizers $\bm v$ at $\z$.
Since $\bm v$ is such a maximizer if and only if
$\z+\bm v\in\mathcal S(\z)$, this is
\eqref{newrand:source-formula}.
\end{proof}
 
To establish the generalized zero-chain property of $\bar f$,
we must control the support of $\partial \fcoup(\y)$
at every maximizing point $\y\in\mathcal S(\z)$.
The following lemma records two properties of these maximizers. 
 
\begin{lemma}
\label{lem:source-order}
For every $\z\in\R^K$ and  $\y\in\mathcal S(\z)$,
the following statements hold:
\begin{enumerate}[label=(\roman*)]
 \item $\y\le\z$ coordinatewise.
 \item If $\y\ne\z$, then
 $
  \partial \fcoup(\y)
  =
\{
  -2\tfrac{\z-\y}{\|\z-\y\|}
  \}.
$
\end{enumerate}
\end{lemma}
\begin{proof}
We first prove (i). Every subgradient of $\phi_m$ is componentwise
nonpositive. Since $w_m\ge0$, $a_m(\bm w,\x)\ge0$, and
$\nabla\psi(\x^m)\ge\bz $,
Lemma~\ref{lem:full-subdiff} implies that
$
 \partial \fcoup(\x)\subseteq\R_-^K$
for all $\x\in\R^K$.
Thus $\fcoup$ is coordinatewise nonincreasing.
Fix $\y\in\mathcal S(\z)$.
If $y_{mn}>z_{mn}$ for some $(m,n)$, replacing
$y_{mn}$ by $z_{mn}$ cannot decrease $\fcoup(\y)$
and strictly decreases $\|\z-\y\|$.
This strictly increases
$\fcoup(\y)-2\|\z-\y\|$, contradicting the maximality
of $\y$.
Hence $\y\le\z$.
 
For (ii), suppose that $\y\ne\z$, so that
$\y'\mapsto\|\z-\y'\|$ is differentiable at $\y$.
Since $\fcoup$ is weakly convex, its directional
derivative is the support function of its
subdifferential \citep[Fact~3]{li2020understanding}.
Optimality of $\y$ in the definition of $\bar f(\z)$
therefore gives
\[
 \max_{\g\in\partial \fcoup(\y)}
 \left\langle
 \g+2\frac{\z-\y}{\|\z-\y\|},\bm d
 \right\rangle
 \le0,
 \qquad \bm d\in\R^K.
\]
Since $\bm d$ is arbitrary and $\partial \fcoup(\y)$
is nonempty, the claimed identity follows. 
\end{proof}
 
We now establish the robust generalized zero-chain property
of $\bar f$: the function value and the full subdifferential
are controlled whenever the unrevealed coordinates are smaller
than $\eta$.
 
\begin{proposition}
\label{newrand:robust}
For every integer $t\ge-1$ and every $\z\in\R^K$
satisfying $|z_{mn}|<\eta$ for all $(m,n)\notin \cI_t$, 
we have
\begin{equation}
\label{newrand:robust-identity}
\begin{aligned}
 \bar f(\z)&=\bar f(\mathbf{\Pi}_{t+1}\z) \quad \text{and } \quad  
 \partial \bar f(\z) =\partial \bar f(\mathbf{\Pi}_{t+1}\z)
 \subseteq \cV_{t+1}.
\end{aligned}
\end{equation}
\end{proposition}
 
\paragraph{Proof Idea.} Before giving the details, we outline the argument.
By \eqref{newrand:source-formula}, the inclusion
$\partial\bar f(\z)\subseteq\mathcal V_{t+1}$ reduces to showing that, for
every $\y\in\mathcal S(\z)$, the subgradients of $\fcoup$ at $\y$
vanish at all coordinates of level greater than $t+1$;  the identities in \eqref{newrand:robust-identity} then follow
from this inclusion.
If $\y=\z$, this is immediate. The unrevealed coordinates of $\z$
are smaller than $\eta$, so within each block only the first
unrevealed affine piece can be active, and the blocks beyond the
last started one have negative scores.
If $\y\ne\z$, we only know that $\y\le\z$: the unrevealed
coordinates of $\y$ may be arbitrarily negative, which could
activate later affine pieces or later blocks.
We show that a maximizing point never does this. Moving the
displacement $z_{mn}-y_{mn}$ from a later unrevealed coordinate to
an earlier one leaves $\|\z-\y\|$ unchanged and strictly increases
$\fcoup$, contradicting $\y\in\mathcal S(\z)$.

Throughout the rest of this subsection, fix $t\ge-1$, a point
$\z\in\R^K$ satisfying $|z_{mn}|<\eta$ for all $(m,n)\notin\cI_t$,
and $\y\in\mathcal S(\z)$. We call the coordinates indexed outside $\cI_t$
unrevealed. Let $k$ denote the index of the last block containing
a coordinate indexed in $\cI_t$, that is,
$k:=\min\{M,1+\lfloor\tfrac{t}{\ell}\rfloor\}$ for $t\ge0$,
and set $k:=1$ for $t=-1$.
Then all coordinates in blocks $m>k$ are unrevealed.
 
\begin{fact}
\label{lem:unrevealed}
The following statements hold:
\begin{enumerate}[label=(\roman*)]
\item For every unrevealed index $(m,n)$, we have
 $q(\tfrac N4y_{mn}-\theta)=q'(\tfrac N4y_{mn}-\theta)=0$.
 \item For unrevealed indices $(m,n)<(m',n')$, we have
 \[c_{mn}-\tfrac N8z_{mn}>c_{m'n'}-\tfrac N8z_{m'n'}.\]
 \item If $k<M$, then $p_m(\y)=0$ for $m\ge k$, and
 $\phi_m(\z^m)<\tfrac54$ for $m>k$.
\end{enumerate}
\end{fact}
\begin{proof}
 (i) By Lemma~\ref{lem:source-order}(i), we have
$y_{mn}\le z_{mn}<\eta<\tfrac{4\theta}{N}$, and
$q$ and $q'$ vanish on $(-\infty,0]$.   (ii) We have $c_{mn}-c_{m'n'}\ge\tfrac1{32(K+1)}$ and
$|z_{mn}-z_{m'n'}|<2\eta$, so the difference of the two sides is
at least $\tfrac1{32(K+1)}-\tfrac{N\eta}{4}>0$. (iii) For $t\ge0$, block $k$ contains exactly
$t-(k-1)\ell+1\le\ell$ coordinates indexed
in $\cI_t$; for $t=-1$, it contains none.
Using (i) and $0\le q\le1$, we obtain
$\psi(\y^k)\le\tfrac{\ell}{N}\le\tfrac14$ and
$\psi(\y^m)=0$ for all $m>k$, which gives the first statement.
For $m>k$, we have $|z_{mn}|<\eta$ for all $n\in[N]$, and hence
$\phi_m(\z^m)<\tfrac{33}{32}+\tfrac{N\eta}{8}<\tfrac54$.
\end{proof}
 
We now turn to the nontrivial case $\y\ne\z$. The next lemma shows
that a maximizing point differs from $\z$ only at coordinates of
level at most $t+1$; all later coordinates are left unchanged.
\begin{lemma}
\label{lem:exchange}
Suppose that $\y\ne\z$, and let $\bm d:=\z-\y$. Then
$d_{mn}=0$ for every $(m,n)\notin\cI_{t+1}$.
\end{lemma}\begin{proof}
By Lemma~\ref{lem:source-order}, we have $\bm d\ge\bz$ and
$\partial \fcoup(\y)=\{-\tfrac{2\bm d}{\|\bm d\|}\}$.
Choose $\bm w\in\mathcal W_\star(\y)$.
For each block $m$, select one term attaining
the maximum defining $\phi_m(\y^m)$, and let
$\bm u^m$ be its gradient.
Then $\bm u^m\in\partial\phi_m(\y^m)$ is either
zero or $-\tfrac N8\bm e_n$ for some $n\in[N]$.
By Lemma~\ref{lem:full-subdiff}, these choices
give a subgradient of $\fcoup$ at $\y$, which must
equal $-\tfrac{2\bm d}{\|\bm d\|}$.
Using Fact~\ref{lem:unrevealed}(i), we obtain
\begin{equation}
\label{eq:exchange-identity}
 -\frac{2d_{mn}}{\|\bm d\|}
 =
 \frac{2w_m}{N}u_{mn}
 \qquad\text{for every unrevealed }(m,n).
\end{equation}

We first show that, in each block, $d_{mn}=0$ for every unrevealed
coordinate after the first one.
Since $\bm u^m$ has at most one nonzero component,
\eqref{eq:exchange-identity} shows that at most one unrevealed
coordinate of $\bm d^m$ can be positive.
Suppose that $d_{mn}>0$ for some unrevealed $(m,n)$.
Then \eqref{eq:exchange-identity} implies that
$w_m>0$ and $\bm u^m=-\tfrac N8\bm e_n$.
Thus
$\phi_m(\y^m)=c_{mn}-\tfrac N8y_{mn}$.
Let $j$ be the smallest index such that
$(m,j)$ is unrevealed. We claim that $n=j$.
If $n>j$, then $d_{mj}=0$.
Define $\widetilde\y$ by
$\widetilde y_{mj}:=z_{mj}-d_{mn}$ and
$\widetilde y_{mn}:=z_{mn}$,
leaving all other coordinates unchanged.
This preserves $\|\z-\y\|$.
Both modified coordinates remain below
$\tfrac{4\theta}{N}$ before and after the modification,
so all progress values and all $p_i$ remain unchanged.
By Fact~\ref{lem:unrevealed}(ii),
\[
 \phi_m(\widetilde\y^m)
 \ge c_{mj}-\frac N8(z_{mj}-d_{mn})
 >c_{mn}-\frac N8(z_{mn}-d_{mn})
 =\phi_m(\y^m).
\]
All other block functions remain unchanged.
Using the same weight $\bm w$, we obtain
\[
 \fcoup(\widetilde\y)
 \ge F_{\bm w}(\widetilde\y)
 =F_{\bm w}(\y)
   +\frac{2w_m}{N}
    \bigl(\phi_m(\widetilde\y^m)-\phi_m(\y^m)\bigr)
 >\fcoup(\y),
\]
where we used $w_m>0$ and $F_{\bm w}(\y)=\fcoup(\y)$.
Together with $\|\z-\widetilde\y\|=\|\z-\y\|$,
this contradicts $\y\in\mathcal S(\z)$.
Hence $n=j$.
 
We next show that $\bm d^m=\bz$ for every $m>k$.
Suppose that $k<M$ and $\bm d^m\ne\bz$ for some $m>k$.
Since every coordinate in this block is unrevealed, the first
part gives
$r:=d_{m1}>0$ and $d_{mn}=0$ for $n=2,\ldots,N$.
By \eqref{eq:exchange-identity} at the coordinate $(m,1)$,
we have $w_m>0$ and $\bm u^m=-\tfrac N8\bm e_1$, so that
$r=\tfrac{\|\bm d\|w_m}{8}$ and
$\phi_m(\y^m)=c_{m1}-\tfrac N8y_{m1}$.
Since this holds for every $\bm w\in\mathcal W_\star(\y)$,
the zero weight does not belong to $\mathcal W_\star(\y)$.
Thus $[\boldsymbol\sigma(\y)]_+\ne\bz$, and
Lemma~\ref{lem:full-subdiff} gives
$\bm w=\tfrac{[\boldsymbol\sigma(\y)]_+}{\|[\boldsymbol\sigma(\y)]_+\|}$;
in particular, $\sigma_m(\y)>0$.
Moreover, $p_{m-1}(\y)=0$ by Fact~\ref{lem:unrevealed}(iii) and
$y_{m1}=z_{m1}-r$, so
\begin{equation}
\label{eq:late-score}
 0<\sigma_m(\y)
 =c_{m1}-\frac N8z_{m1}-\frac54+\frac N8r
 <\frac N8r,
\end{equation}
where the last inequality follows from
$c_{m1}<\tfrac{33}{32}$ and $|z_{m1}|<\eta$.
 
We now move the displacement $r$ from the coordinate $(m,1)$
to a coordinate of block $k$. We choose an index $n\in[N]$ with
$y_{kn}<\tfrac{4\theta}{N}$ as follows.
\begin{enumerate}[label=(\alph*)]
 \item If $\phi_k(\y^k)>\tfrac54$, let $n$ be the index of an
 affine piece attaining $\phi_k(\y^k)$. Then
 $c_{kn}-\tfrac N8y_{kn}>\tfrac54>c_{kn}$, so $y_{kn}<0$.
 \item If $\phi_k(\y^k)\le\tfrac54$, let $(k,n)$ be any unrevealed
 coordinate of block $k$, which exists because block $k$ contains
 at most $\ell<N$ coordinates indexed in $\cI_t$. Then
 $y_{kn}\le z_{kn}<\eta<\tfrac{4\theta}{N}$.
\end{enumerate}
Define $\widetilde\y$ by
$\widetilde y_{kn}:=z_{kn}-\sqrt{d_{kn}^2+r^2}$
and $\widetilde y_{m1}:=z_{m1}$,
leaving all other coordinates unchanged.
The two displacements $(d_{kn},r)$ are replaced
by $(\sqrt{d_{kn}^2+r^2},0)$, so
$\|\z-\widetilde\y\|=\|\z-\y\|$.
Moreover,
$\widetilde y_{kn}\le y_{kn}<\tfrac{4\theta}{N}$
and
$y_{m1}\le\widetilde y_{m1}=z_{m1}
<\eta<\tfrac{4\theta}{N}$,
so all progress values and all $p_i$
remain unchanged.
Since $\widetilde\y^m=\z^m$ and $p_{m-1}(\widetilde\y)=0$,
Fact~\ref{lem:unrevealed}(iii) gives
$\sigma_m(\widetilde\y)=\phi_m(\z^m)-\tfrac54<0$,
and all scores other than $\sigma_k$ and $\sigma_m$ are unchanged.
We claim that
\begin{equation}
\label{eq:exchange-gain}
 \sigma_k(\widetilde\y)
 >\sqrt{[\sigma_k(\y)]_+^2+\sigma_m(\y)^2}.
\end{equation}
Then
$\|[\boldsymbol\sigma(\widetilde\y)]_+\|
>\|[\boldsymbol\sigma(\y)]_+\|$,
and hence $\fcoup(\widetilde\y)>\fcoup(\y)$ since all progress
values are unchanged.
Together with $\|\z-\widetilde\y\|=\|\z-\y\|$,
this contradicts $\y\in\mathcal S(\z)$, and therefore
$\bm d^m=\bz$ for every $m>k$.
 
It remains to prove \eqref{eq:exchange-gain}.
In case (a), we have
$\sigma_k(\y)=\phi_k(\y^k)+p_{k-1}(\y)-\tfrac54>0$,
and the selected piece has gradient
$\bm u^k=-\tfrac N8\bm e_n$.
Since $y_{kn}<0$, the corresponding progress
derivative satisfies
$(\nabla\psi(\y^k))_n=0$.
Choosing this $\bm u^k$ in
Lemma~\ref{lem:full-subdiff} and using
Lemma~\ref{lem:source-order}(ii), we obtain
\begin{equation}
\label{eq:2}
 d_{kn}=\frac{\|\bm d\|w_k}{8},
 \qquad
 \frac{\sigma_k(\y)}{\sigma_m(\y)}
 =\frac{w_k}{w_m}=\frac{d_{kn}}{r}.
\end{equation}
The modification decreases $y_{kn}$ and leaves
all other coordinates of block $k$ unchanged.
Thus the selected affine piece increases,
while all other pieces remain unchanged, and it
therefore remains active at $\widetilde\y^k$.
Using \eqref{eq:2} and $\tfrac{\sigma_m(\y)}{r}<\tfrac N8$
from \eqref{eq:late-score}, we obtain
\[
\begin{aligned}
 \sigma_k(\widetilde\y)
 &=\frac{\sigma_m(\y)}{r}d_{kn}
   +\frac N8\left(\sqrt{d_{kn}^2+r^2}-d_{kn}\right)\\
 &>\frac{\sigma_m(\y)}{r}\sqrt{d_{kn}^2+r^2}
 =\sqrt{\sigma_k(\y)^2+\sigma_m(\y)^2}.
\end{aligned}
\]
 
In case (b), since
$\sigma_k(\y)=\phi_k(\y^k)+p_{k-1}(\y)-\tfrac54
\le p_{k-1}(\y)$
and $p_{k-1}(\y)\ge0$, we have
$[\sigma_k(\y)]_+\le p_{k-1}(\y)$.
Moreover, the definition of $\phi_k$ gives
$\phi_k(\widetilde\y^k)
\ge c_{kn}-\tfrac N8\widetilde y_{kn}$,
and the modification preserves $p_{k-1}$.
Since $(k,n)$ and $(m,1)$ are unrevealed and $k<m$,
Fact~\ref{lem:unrevealed}(ii),
$\sqrt{d_{kn}^2+r^2}\ge r$, and \eqref{eq:late-score} give
\[
\begin{aligned}
 \sigma_k(\widetilde\y)
 &\ge c_{kn}-\frac N8z_{kn}
   +\frac N8\sqrt{d_{kn}^2+r^2}
   +p_{k-1}(\y)-\frac54\\
 &>c_{m1}-\frac N8z_{m1}
   +\frac N8r+p_{k-1}(\y)-\frac54\\
 &=\sigma_m(\y)+p_{k-1}(\y)
 \ge\sigma_m(\y)+[\sigma_k(\y)]_+
 \ge\sqrt{\sigma_m(\y)^2+[\sigma_k(\y)]_+^2},
\end{aligned}
\]
where the last inequality follows from
$\sigma_m(\y)>0$ and $[\sigma_k(\y)]_+\ge0$.
This proves \eqref{eq:exchange-gain}.
 
Finally, in each block $m\le k$, the first unrevealed coordinate,
if it exists, has level $t+1$. By the two parts above, the vector
$\bm d$ vanishes at all other unrevealed coordinates, that is,
outside $\cI_{t+1}$.
\end{proof}
 
We can now prove Proposition~\ref{newrand:robust}.
\begin{proof}
We first prove that $\partial \bar f(\z)\subseteq\cV_{t+1}$.
By \eqref{newrand:source-formula}, it suffices to show that
$\partial \fcoup(\y)\subseteq\cV_{t+1}$.
 
If $\y\ne\z$, Lemma~\ref{lem:source-order}(ii) gives
$\partial \fcoup(\y)=\{-\tfrac{2\bm d}{\|\bm d\|}\}$ with
$\bm d=\z-\y$, and the inclusion follows from
Lemma~\ref{lem:exchange}.
 
Suppose now that $\y=\z$, and fix $\g\in\partial\fcoup(\z)$.
By Fact~\ref{lem:unrevealed}(ii), within each
block, no affine piece after the first unrevealed one
can attain the maximum defining
$\phi_m(\z^m)$, so every subgradient of $\phi_m$ vanishes
at those subsequent coordinates.
Together with Fact~\ref{lem:unrevealed}(i),
Lemma~\ref{lem:full-subdiff} shows that $\g$ vanishes at every
unrevealed coordinate after the first one in its block.
Moreover, if $k<M$, then for every $m>k$,
Fact~\ref{lem:unrevealed}(iii) gives
$\sigma_m(\z)=\phi_m(\z^m)-\tfrac54<0$.
Thus every $\bm w\in\mathcal W_\star(\z)$ satisfies
$w_m=0$, and since $\nabla\psi(\z^m)=\bz$ by
Fact~\ref{lem:unrevealed}(i),
Lemma~\ref{lem:full-subdiff} gives $\g^m=\bz$.
As in the proof of Lemma~\ref{lem:exchange}, the remaining
unrevealed coordinates have level $t+1$.
Hence $\g\in\cV_{t+1}$.
 
It remains to prove
$\bar f(\z)=\bar f(\mathbf{\Pi}_{t+1}\z)$ and
$\partial \bar f(\z)=\partial \bar f(\mathbf{\Pi}_{t+1}\z)$.
For every $\bm x$ satisfying the stated coordinate
bounds and every $\bm a\in\cV_{t+1}^{\perp}$,
weak convexity and the support inclusion give
$ \bar f'(\bm x;\bm a)
 =
 \max_{\g\in\partial \bar f(\bm x)}
 \langle\g,\bm a\rangle
 =0.
$
Thus changing only the coordinates outside
$\cI_{t+1}$, while preserving the coordinate bounds,
leaves $\bar f$ unchanged. In particular,
$\bar f(\z+\bm v)=\bar f(\mathbf{\Pi}_{t+1}\z+\bm v)$
for all sufficiently small $\bm v$.
Taking $\bm v=\bz$ gives the function value identity.
Taking subdifferentials with respect to $\bm v$
at zero gives the subdifferential identity.
This completes the proof.
\end{proof}

\subsection{The terminal layer stationarity bound} 
We next verify \textnormal{(C3)}.
\begin{proposition}
\label{prop:terminal-stationarity}
The horizon $T=(M-1)\ell$ satisfies $T\ge\tfrac{MN}{16}$,
and every $\z\in\R^K$ satisfying $|z_{mn}|<\eta$ for all
$(m,n)\notin\cI_{T-1}$ obeys
\[
 \left\|\nabla \bar f_{\tfrac19}(\z)\right\|
 >\frac1{10\sqrt N}.
\]
\end{proposition}
The proof relies on a lower bound on the subgradients of $\bar f$,
which is also used in Section~\ref{sec:proof}. Let $r_0:=\tfrac1{360\sqrt N}$ and
\[
 \mathcal Z
 :=
 \left\{
 \z\in\R^K\mid
 \sigma_m(\z)\le0
 \text{ for every }m\in[M]
 \right\}.
\]
The set $\mathcal Z$ consists of the points at which no block is active.
Since $p_0\equiv1$, these are the points at which the whole chain has been
completed. The next lemma shows that points whose
last block is unrevealed are far from $\mathcal Z$, and that
$\bar f$ has no small subgradients away from $\mathcal Z$.
\begin{lemma}
\label{lem:away-from-Z}
The following statements hold:
\begin{enumerate}[label=(\roman*)]
 \item We have $\mathcal Z\subset(\tfrac6N,+\infty)^K$. Consequently,
 $\operatorname{dist}(\z,\mathcal Z)\ge\tfrac5{\sqrt N}$ for every
 $\z\in\R^K$ with at least $N$ coordinates not exceeding $\eta$.
 \item For every $\z\in\R^K$ and $\y\in\mathcal S(\z)$, we have
 $\|\z-\y\|<\tfrac{3M}{N}$.
 \item If $\operatorname{dist}(\z,\mathcal Z)>r_0$, then
 $\|\g\|\ge\tfrac1{\sqrt{98N}}$ for every $\g\in\partial\bar f(\z)$.
\end{enumerate}
\end{lemma}
\begin{proof}
(i) Fix $\x\in\mathcal Z$.
Starting from $p_0(\x)=1$, we proceed by induction
over the blocks.
Suppose that $p_{m-1}(\x)=1$.
Since $\sigma_m(\x)\le0$, we have
\[
 \phi_m(\x^m)\le\frac14,
 \qquad
 x_{mn}\ge\frac8N\left(c_{mn}-\frac14\right)
 >\frac6N,
 \qquad n\in[N].
\]
Thus $\tfrac N4x_{mn}-\theta>1$ for every $n\in[N]$,
so the definitions of $\psi$ and $p_m$ give
$\psi(\x^m)=p_m(\x)=1$.
This completes the induction and proves
$\mathcal Z\subset(\tfrac6N,+\infty)^K$.
Finally, if $\z\in\R^K$ has at least $N$ coordinates not exceeding
$\eta\le\tfrac1N$, then
$\|\z-\x\|\ge\sqrt N(\tfrac6N-\eta)\ge\tfrac5{\sqrt N}$
for every $\x\in\mathcal Z$.
 
(ii) Each $\phi_m$ is $\tfrac N8$-Lipschitz.
Moreover, the definition of $p_m$ and
Lemma~\ref{lem:progress-regularity} give
$ \|\nabla p_m(\x)\|
 \le2\|\nabla\psi(\x^m)\|
 \le\tfrac{3\sqrt N}{4}.
$
Thus, for every $\bm w\in\mathcal W$ and
$\x,\x'\in\R^K$,
\[
\begin{aligned}
 &\left|
 \frac2N\sum_{m=1}^M
 w_m\bigl(\sigma_m(\x)-\sigma_m(\x')\bigr)
 \right|\\
 \quad\le &
 \frac14\sum_{m=1}^M
 w_m\|\x^m-\x'^m\|
 +\frac3{2\sqrt N}\sum_{m=1}^{M-1}
 w_{m+1}\|\x^m-\x'^m\|\\
\le & \left(\frac14+\frac3{2\sqrt N}\right)
 \|\x-\x'\| \leq \frac{11}{32} \|\x-\x'\| ,
\end{aligned}
\]
where the second inequality follows from
Cauchy--Schwarz and $\|\bm w\|\le1$, and the last uses $N\ge256$.
Since $\|[\a]_+\|=\max_{\bm w\in\mathcal W}\langle\bm w,\a\rangle$,
taking the maximum over $\bm w$ gives
\begin{equation}
 \frac2N
 \left|
 \|[\boldsymbol\sigma(\x)]_+\|
 -
 \|[\boldsymbol\sigma(\x')]_+\|
 \right|
 \le\frac{11}{32}\|\x-\x'\|.
 \label{eq:score-lipschitz}
\end{equation}
Now fix $\z\in\R^K$ and $\y\in\mathcal S(\z)$, so that
$\fcoup(\y)-2\|\z-\y\|=\bar f(\z)\ge \fcoup(\z)$.
By \eqref{eq:score-lipschitz} and $0\le\psi\le1$,
\[
 2\|\z-\y\|
 \le \fcoup(\y)-\fcoup(\z)
 \le\frac{11}{32}\|\z-\y\|+\frac{4M}{N},
\]
and hence $\|\z-\y\|\le\tfrac{128M}{53N}<\tfrac{3M}{N}$.
 
(iii) We first bound the subgradients of $\fcoup$ outside
$\mathcal Z$. Fix $\x\notin\mathcal Z$.
Then $[\boldsymbol\sigma(\x)]_+\ne\bz$, and by
Lemma~\ref{lem:full-subdiff} the
unique active weight is $ \bm w
 =\tfrac{[\boldsymbol\sigma(\x)]_+}{\|[\boldsymbol\sigma(\x)]_+\|}$.
For every $m$ with $w_m>0$, we have
$\sigma_m(\x)>0$.
Since $p_{m-1}(\x)\le1$, it follows that
\[
 \phi_m(\x^m)
 =\sigma_m(\x)-p_{m-1}(\x)+\frac54
 >\frac14.
\]
Thus the zero branch of $\phi_m$ is inactive.
Every $\bm u^m\in\partial\phi_m(\x^m)$ is therefore
a convex combination of vectors
$-\tfrac N8\bm e_n$, and hence $ \bm u^m\le\bz$,
$\|\bm u^m\|_1=\tfrac N8$, and $\|\bm u^m\|
 \ge\tfrac{\|\bm u^m\|_1}{\sqrt N}
 =\tfrac{\sqrt N}{8}$.
By Lemma~\ref{lem:full-subdiff}, every
$\g\in\partial \fcoup(\x)$ has block components
\[
 \g^m
 =
 \frac2N
 \bigl(
 w_m\bm u^m
 -
 a_m(\bm w,\x)\nabla\psi(\x^m)
 \bigr),
 \qquad
 \bm u^m\in\partial\phi_m(\x^m).
\]
Since $\bm u^m\le\bz$,
$a_m(\bm w,\x)\ge0$, and
$\nabla\psi(\x^m)\ge\bz$, the two terms are
componentwise nonpositive. Consequently, we have $\|\g^m\|
 \ge\tfrac{2w_m}{N}\|\bm u^m\|
 \ge\tfrac{w_m}{4\sqrt N}$  whenever $w_m>0$.
Summing over the blocks and using $\|\bm w\|=1$ gives
\begin{equation}
 \|\g\|\ge\frac1{4\sqrt N}
 \qquad\text{for every }
 \x\notin\mathcal Z\text{ and }
 \g\in\partial \fcoup(\x).
 \label{eq:raw-stationarity}
\end{equation}
 
Now fix $\z\in\R^K$ with $\operatorname{dist}(\z,\mathcal Z)>r_0$.
By \eqref{newrand:source-formula}, every element of
$\partial \bar f(\z)$ is a convex combination of subgradients of
$\fcoup$ at points of $\mathcal S(\z)$.
Since a convex combination can be much shorter than each of its
terms, \eqref{eq:raw-stationarity} alone is not enough.
We group the points of $\mathcal S(\z)$ by their distance from
$\z$ and show that subgradients within a group are nearly aligned.
Since all subgradients are componentwise nonpositive, different
groups cannot cancel each other.
 
We first consider the maximizing points close to $\z$, and show
that every convex combination of subgradients
at points $\y\in\mathcal S(\z)$ with
$\|\z-\y\|\le r_0$ has norm at least
$\tfrac1{5\sqrt N}$.
To average the weak convexity inequalities at different maximizing
points, we first rewrite each of them around the point $\z$.
Fix $\y\in\mathcal S(\z)$ and
$\g\in\partial \fcoup(\y)$.
When $\y\ne\z$, Lemma~\ref{lem:source-order}(ii)
gives $\langle\g,\z-\y\rangle=-2\|\z-\y\|$,
which  also holds when $\y=\z$.
Thus, by the definition of $\mathcal S(\z)$,
\[
 \fcoup(\y)+\langle\g,\z-\y\rangle
 =\fcoup(\y)-2\|\z-\y\|
 =\bar f(\z).\]
Since $\fcoup$ is $\tfrac92$-weakly convex,
for every $\bm v\in\R^K$,
\begin{equation}
\label{eq:wc-around-z}
\begin{aligned}
 \fcoup(\bm v)
 &\ge
 \fcoup(\y)+\langle\g,\bm v-\y\rangle
 -\frac94\|\bm v-\y\|^2\\
 &=\fcoup(\y)+\langle\g,\z-\y\rangle
 +\langle\g,\bm v-\z\rangle
 -\frac94\|\bm v-\y\|^2\\
 &=\bar f(\z)+\langle\g,\bm v-\z\rangle
 -\frac94\|\bm v-\y\|^2.
\end{aligned}
\end{equation}
Consider a finite convex combination $ \bar{\g}=\sum_i\lambda_i\g_i$,
with $\lambda_i\ge0$ and $\sum_i\lambda_i=1$,
where $\g_i\in\partial \fcoup(\y_i)$,
$\y_i\in\mathcal S(\z)$, and
$\|\z-\y_i\|\le r_0$.
Averaging \eqref{eq:wc-around-z} over $i$ and using
$\|\bm v-\y_i\|\le\|\bm v-\z\|+r_0$ gives
\begin{equation}
\label{eq:wc-averaged}
 \fcoup(\bm v)
 \ge \bar f(\z)+\langle\bar{\g},\bm v-\z\rangle
 -\frac94\bigl(\|\bm v-\z\|+r_0\bigr)^2.
\end{equation}
Let $\bm v_*$ minimize
$
 \fcoup(\bm v)-\langle\bar{\g},\bm v-\z\rangle
 +9\|\bm v-\z\|^2.
$
By optimality of $\bm v_*$,
\[
 \fcoup(\bm v_*)-\langle\bar{\g},\bm v_*-\z\rangle
 +9\|\bm v_*-\z\|^2
 \le \fcoup(\z)\le \bar f(\z).
\]
On the other hand, \eqref{eq:wc-averaged} at $\bm v=\bm v_*$ gives
\[
 \fcoup(\bm v_*)-\langle\bar{\g},\bm v_*-\z\rangle
 \ge \bar f(\z)
 -\frac94\bigl(\|\bm v_*-\z\|+r_0\bigr)^2.
\]
Combining these inequalities, we obtain
$
 9\|\bm v_*-\z\|^2
 \le\tfrac94\bigl(\|\bm v_*-\z\|+r_0\bigr)^2.
$
Thus $\|\bm v_*-\z\|\le r_0$, so
$\bm v_*\notin\mathcal Z$.
The optimality condition gives
$
 \bar{\g}-18(\bm v_*-\z)
 \in\partial \fcoup(\bm v_*).
$
By \eqref{eq:raw-stationarity},
$
 \tfrac1{4\sqrt N}
 \le\|\bar{\g}-18(\bm v_*-\z)\|
 \le\|\bar{\g}\|+18r_0.
$
Consequently,
$\|\bar{\g}\|
 \ge\tfrac1{4\sqrt N}-18r_0
 =\tfrac1{5\sqrt N}$.
 
We next consider the maximizing points far from $\z$, and show
that, for every $a\ge r_0$, every convex combination of
subgradients at maximizing points whose distances from $\z$
lie in $(a,2a]$ has norm at least $\tfrac13$.
Fix $\y,\y'\in\mathcal S(\z)$ such that
$
 r_0<r:=\|\z-\y\|
 \le r':=\|\z-\y'\|\le2r,
$
and set $\bm d:=\z-\y\ge\bz$ and
$\bm d':=\z-\y'\ge\bz$ by Lemma~\ref{lem:source-order}(i).
Moreover, by Lemma~\ref{lem:source-order}(ii), their
subdifferentials are singletons, with
$
 \g=-\tfrac{2\bm d}{r}$ and
$\g'=-\tfrac{2\bm d'}{r'}.
$
 
We compare $\y$ and $\y'$ with the points $\y\wedge\y'$ and
$\y\vee\y'$. Since $\z-(\y\wedge\y')=\bm d\vee\bm d'$ and
$\z-(\y\vee\y')=\bm d\wedge\bm d'$, the definition of $\bar f(\z)$
gives
$\bar f(\z)\ge\fcoup(\y\wedge\y')-2\|\bm d\vee\bm d'\|$ and
$\bar f(\z)\ge\fcoup(\y\vee\y')-2\|\bm d\wedge\bm d'\|$.
Adding these inequalities and using
$\bar f(\z)=\fcoup(\y)-2r=\fcoup(\y')-2r'$, we obtain
\[
 2\bigl(
 r+r'
 -\|\bm d\vee\bm d'\|
 -\|\bm d\wedge\bm d'\|
 \bigr)\le
 \fcoup(\y)+\fcoup(\y')
 -\fcoup(\y\wedge\y')-\fcoup(\y\vee\y').
\]
We now bound the right-hand side. The progress terms cancel, since
$\y\wedge\y'$ and $\y\vee\y'$ have the same coordinates as $\y$ and
$\y'$ up to exchanging them. The right-hand side therefore equals
$
 \tfrac2N(\|[\boldsymbol\sigma(\y)]_+\|
   -\|[\boldsymbol\sigma(\y\vee\y')]_+\|)
 +\tfrac2N(\|[\boldsymbol\sigma(\y')]_+\|
   -\|[\boldsymbol\sigma(\y\wedge\y')]_+\|).
$
Since
$\|\y-(\y\vee\y')\|=\|\y'-(\y\wedge\y')\|=\|[\bm d-\bm d']_+\|\le r$,
\eqref{eq:score-lipschitz} bounds each of the two terms by
$\tfrac{11}{32}r$. Consequently,
\begin{equation}
\label{eq:common-part}
 \|\bm d\wedge\bm d'\|
 \ge r+r'-\|\bm d\vee\bm d'\|-\frac{11}{32}r.
\end{equation}
Since $\bm d,\bm d'\ge\bz$,
we have $\|\bm d\vee\bm d'\|^2\le\|\bm d\|^2+\|\bm d'\|^2=r^2+r'^2$,
and the triangle inequality applied to $(r,r')=(r,r)+(0,r'-r)$ gives
$
 \|\bm d\vee\bm d'\|
 \le\sqrt{r^2+r'^2}
 \le\sqrt2\,r+(r'-r).
$ 
Substituting this into \eqref{eq:common-part}, we obtain
\[
 \|\bm d\wedge\bm d'\|
 \ge\Bigl(2-\sqrt2-\frac{11}{32}\Bigr)r
 =\Bigl(\frac{53}{32}-\sqrt2\Bigr)r
 >\frac{r}{3\sqrt2}.
\]
Moreover, since $\bm d,\bm d'\ge\bz$, we have
$\langle\bm d,\bm d'\rangle\ge\|\bm d\wedge\bm d'\|^2$.
Together with $r'\le2r$, this gives
\[
 \langle\g,\g'\rangle
 =
 \frac{4\langle\bm d,\bm d'\rangle}{rr'}
 \ge
 \frac{4\|\bm d\wedge\bm d'\|^2}{rr'}
 >\frac{4r^2/18}{2r^2}
 =\frac19.
\]
Thus, for any finite convex combination
$\bar{\g}=\sum_i\lambda_i\g_i$
whose maximizing points have distances from $\z$
lying in a common interval $(a,2a]$ with $a\ge r_0$, we have
$\|\bar{\g}\|^2=\sum_{i,j}\lambda_i\lambda_j\langle\g_i,\g_j\rangle
\ge\tfrac19$, and hence $\|\bar{\g}\|\ge\tfrac13$.
 
 Finally, we combine the bounds for the different groups. Averaging
over groups may reduce the norm by a factor depending on the number
of groups, so we first bound this number. 
By $M\le2^{8N}$ and $N\ge256$,
\[
 \frac{3M/N}{r_0}
 =\frac{1080M}{\sqrt N}
 \le\frac{135}{2}\,2^{8N}
 <2^{8N+7}.
\]
Thus (ii) implies
that all maximizing points are covered by the
distance intervals $[0,r_0]$ and $(2^{j-1}r_0,2^jr_0]$ for $j=1,\ldots,8N+7$.
 
Fix $\g\in\partial \bar f(\z)$.
By \eqref{newrand:source-formula}, $\g$ is a finite
convex combination of subgradients at maximizing
points.
Grouping these terms by the intervals above gives
\[
 \g=\sum_{j=0}^{8N+7}\lambda_j\bar{\g}_j,
 \qquad
 \lambda_j\ge0,\quad
 \sum_{j=0}^{8N+7}\lambda_j=1,
\]
where each $\bar{\g}_j$ with $\lambda_j>0$
is a convex combination within the corresponding
group; set $\bar{\g}_j=\bz$ when $\lambda_j=0$.
 By Lemma~\ref{lem:full-subdiff}, all these vectors
are componentwise nonpositive, so their pairwise
inner products are nonnegative. Hence
$\|\g\|^2\ge\sum_j\lambda_j^2\|\bar\g_j\|^2$, and the two bounds
above give
\[
 \|\g\|^2
 \ge
 \frac{\lambda_0^2}{25N}
 +\frac19\sum_{j=1}^{8N+7}\lambda_j^2\ge
 \frac{\bigl(\sum_{j=0}^{8N+7}\lambda_j\bigr)^2}
 {25N+9(8N+7)}
 =\frac1{97N+63}
 \ge\frac1{98N},
\]
where the second inequality is the Cauchy--Schwarz inequality.
This proves (iii).
\end{proof}
We can now prove Proposition~\ref{prop:terminal-stationarity}. 
\begin{proof}
Since $M-1\ge\tfrac M2$ and $\ell\ge\tfrac N8$,
we have $T\ge \tfrac{MN}{16}$.
Fix $\z\in\R^K$ satisfying $|z_{mn}|<\eta$ for all
$(m,n)\notin\cI_{T-1}$.
Every coordinate in block $M$ has level at least $T$
and hence lies outside $\cI_{T-1}$, so
$z_{Mn}<\eta$ for every $n\in[N]$.
Lemma~\ref{lem:away-from-Z}(i) therefore gives
$\operatorname{dist}(\z,\mathcal Z)\ge\tfrac5{\sqrt N}$.
 
By \citet[Lemma~2.2]{DavisDrusvyatskiy2019}, we have
$
 \nabla\bar f_{\tfrac19}(\z)
 =9\bigl(\z-\prox_{\tfrac19\bar f}(\z)\bigr)
 \in\partial\bar f\bigl(\prox_{\tfrac19\bar f}(\z)\bigr).
$
If $\|\z-\prox_{\tfrac19\bar f}(\z)\|>\tfrac1{90\sqrt N}$, then
$\|\nabla \bar f_{\tfrac19}(\z)\|>\tfrac1{10\sqrt N}$.
Otherwise, the triangle inequality gives
\[
 \operatorname{dist}\left(\prox_{\tfrac19\bar f}(\z),\mathcal Z\right)
 \ge\operatorname{dist}(\z,\mathcal Z)
      -\left\|\z-\prox_{\tfrac19\bar f}(\z)\right\|
 \ge\frac5{\sqrt N}-\frac1{90\sqrt N}
 >r_0,
\]
and Lemma~\ref{lem:away-from-Z}(iii), applied at
$\prox_{\tfrac19\bar f}(\z)$, yields
$\|\nabla \bar f_{\tfrac19}(\z)\|
 \ge\tfrac1{\sqrt{98N}}
 >\tfrac1{10\sqrt N}$. This completes the proof.
\end{proof}

\subsection{Initial gap}
\label{sec:gap} 
 
We finally verify \textnormal{(C4)}.
\begin{proposition}
\label{newrand:gap}
We have $ \bar f(\bz)-\inf \bar f\le\tfrac{5M}{N}$.
\end{proposition}
\begin{proof}
Let $\y\in\mathcal S(\bz)$.
By Lemma~\ref{lem:source-order}(i), we have $\y\le\bz$.
Since $q$ vanishes on $(-\infty,0]$, we have  $ \psi(\y^m) =0$ and $ p_m(\y)=0$ for all $m\in [M]$,
and the same holds at $\bz$.
Each affine piece defining $\phi_m$ has slope norm
at most $\tfrac N8$, so $\phi_m$ is
$\tfrac N8$-Lipschitz. Since the progress terms of $\fcoup$ vanish
at $\y$ and $\bz$, and $\a\mapsto\|[\a]_+\|$ is $1$-Lipschitz, we have
\[
\begin{aligned}
 |\fcoup(\y)-\fcoup(\bz)|
 &\le\frac2N
 \|\boldsymbol\sigma(\y)-\boldsymbol\sigma(\bz)\|=\frac2N
 \left(
 \sum_{m=1}^M
 |\phi_m(\y^m)-\phi_m(\bz)|^2
 \right)^{1/2}\\
 &\le\frac2N\cdot\frac N8
 \left(\sum_{m=1}^M\|\y^m\|^2\right)^{1/2}
 =\frac14\|\y\|.
\end{aligned}
\]
Consequently,
\[
 \bar f(\bz)
 =\fcoup(\y)-2\|\y\|
 \le \fcoup(\bz)-\frac74\|\y\|
 \le \fcoup(\bz).
\]
Together with $\bar f\ge \fcoup$, this gives
$\bar f(\bz)=\fcoup(\bz)$.
 
At the origin,
$\sigma_1(\bz)=c_{11}-\tfrac14>0$ and
$\sigma_m(\bz)=c_{m1}-\tfrac54<0$ for $m\ge2$.
Hence
$
 \bar f(\bz)=\fcoup(\bz)=\tfrac2N(c_{11}-\tfrac14).
$ 
Finally, \eqref{eq:lower_f} and $\bar f\ge \fcoup$ give
$\inf \bar f\ge-\tfrac{4M}{N}$.
Since $c_{11}-\tfrac14<1$ and $M\ge2$,
\[
 \bar f(\bz)-\inf \bar f
 \le\frac2N\left(2M+c_{11}-\frac14\right)
 <\frac{4M+2}{N}
 \le\frac{5M}{N}.
\]
\end{proof}

\section{Proof of Theorem~\ref{thm:main}}
\label{sec:proof}
Starting from the hard instance $\bar f$ constructed above,
we establish the lower bound for arbitrary randomized
first-order algorithms through a radial modification and a
random rotation and rescaling argument.
\subsection{The radial modification}
\label{sec:radial-verification} 
We verify the counterparts of \textnormal{(C1)--(C4)}
for $h_{\mathbf U}$ before applying the random rotation
argument.
Recall from \eqref{eq:radial-instance} that
\[
 h_{\mathbf{U}}(\bm z)
 :=
 \max\left\{
 \frac12\bar f\left(\frac23\mathbf{U}^\top\bm z\right),
 \frac12\bar f(\bz)+\|\bm z\|-R
 \right\},
\]
where $\mathbf U^\top\mathbf U=\mathbf I_K$
and $R>0$ is specified in Lemma~\ref{lem:radial-stationarity}.

By \textnormal{(C1)} and $\|\mathbf U^\top\|=1$, the first function
in the maximum, $\bm z\mapsto\tfrac12\bar f(\tfrac23\mathbf U^\top\bm z)$,
is $\tfrac23$-Lipschitz and $1$-weakly convex.
Since the second function, the radial term,
is convex and $1$-Lipschitz, $h_{\mathbf U}$ is
finite-valued, bounded below, $1$-Lipschitz,
and $1$-weakly convex.  Moreover, $h_{\mathbf U}(\bz)=\tfrac12\bar f(\bz)$
and $\inf h_{\mathbf U}\ge\tfrac12\inf\bar f$,
so \textnormal{(C4)} gives
\[
 h_{\mathbf U}(\bz)-\inf h_{\mathbf U}
 \le\frac12\bigl(\bar f(\bz)-\inf\bar f\bigr)
 \le\frac{5M}{2N}\le\frac{3M}{N}.
\]
 
To control the gradient of the Moreau envelope of
$h_{\mathbf U}$ using the subgradient bound for
$\bar f$, we need to exclude cancellations between the subgradients
of the two pieces in \eqref{eq:radial-instance}. To this end,
we first establish
\begin{equation}
 \langle\g,\z\rangle
 \ge-\frac{18\sqrt K+6M}{N},
 \qquad
 \z\in\R^K,\quad \g\in\partial\bar f(\z).
 \label{eq:radial-inner-product}
\end{equation}
Fix $\y\in\mathcal S(\z)$ and $\g\in\partial \fcoup(\y)$.
By \eqref{newrand:source-formula}, we have $\g\in\partial\bar f(\z)$,
so $\|\g\|\le2$ by \textnormal{(C1)}; moreover, $\g\le\bz$ by
Lemma~\ref{lem:full-subdiff}.
For every coordinate with $y_{mn}>\tfrac9N$, we have
$c_{mn}-\tfrac N8y_{mn}<0$ and $q'(\tfrac N4y_{mn}-\theta)=0$,
so the same lemma gives $g_{mn}=0$.
Hence $g_{mn}y_{mn}\ge\tfrac9N g_{mn}$ for every coordinate, and
therefore $\langle\g,\y\rangle\ge-\tfrac9N\|\g\|_1$.
Since $\|\g\|_1\le\sqrt K\|\g\|\le2\sqrt K$ and
$\|\z-\y\|<\tfrac{3M}{N}$ by Lemma~\ref{lem:away-from-Z}(ii),
we obtain
\[
 \langle\g,\z\rangle
=\langle\g,\y\rangle+\langle\g,\z-\y\rangle
 \ge-\frac9N\|\g\|_1-\|\g\|\|\z-\y\|\ge-\frac{18\sqrt K+6M}{N}.
\]
Taking convex combinations and applying
\eqref{newrand:source-formula} proves
\eqref{eq:radial-inner-product}.

Using \eqref{eq:radial-inner-product}, we now establish
the lower bound on the gradient of the Moreau envelope
of $h_{\mathbf U}$.
\begin{lemma}[Properties of the radial modification]
\label{lem:radial-stationarity}
Set $R:=2^{18}(M+\sqrt K)$.
For  every
$\mathbf U\in\R^{d\times K}$ satisfying
$\mathbf U^\top\mathbf U=\mathbf I_K$, the function
$h_{\mathbf U}$ in \eqref{eq:radial-instance}
has the following properties.
\begin{enumerate}[label=(\roman*)]
 \item If $\|\bm z\|\ge4R$, then
 \[
  \mathbb O_{h_{\mathbf U}}(\bm z)
  =\left(
   \tfrac12\bar f(\bz)+\|\bm z\|-R,
   \left\{\frac{\bm z}{\|\bm z\|}\right\}
  \right), \quad \text{and} \quad \left\|\nabla(h_{\mathbf U})_{\tfrac12}(\bm z )\right\|=1.
 \]
 \item If at least $N$ coordinates of
 $\tfrac23\mathbf U^\top\bm z$ do not exceed $\eta$, then
 \[
  \left\|\nabla(h_{\mathbf U})_{\tfrac12}(\bm z)\right\|
  >\frac1{32\sqrt N}.
 \]
\end{enumerate}
\end{lemma}
\begin{proof}
By \citet[Lemma~2.2]{DavisDrusvyatskiy2019}
and proximal optimality,
\begin{equation}
\label{eq:prox}
 \nabla(h_{\mathbf U})_{\tfrac12}(\bm z)
 =2\left(
 \bm z-\prox_{\tfrac12 h_{\mathbf U}}(\bm z)
 \right)
 \in\partial h_{\mathbf U}\left(
 \prox_{\tfrac12 h_{\mathbf U}}(\bm z)
 \right),
\end{equation}
with $\|\bm z-\prox_{\tfrac12 h_{\mathbf U}}(\bm z)\|
\le\tfrac12$ since $h_{\mathbf U}$ is $1$-Lipschitz.
 
\textrm{(i)}
Since $\bar f$ is $2$-Lipschitz, for every $\bm z$ with
$\|\bm z\|>3R$,
\[
 \frac12\bar f\left(\frac23\mathbf U^\top\bm z\right)
 \le\frac12\bar f(\bz)+\frac23\|\bm z\|
 <\frac12\bar f(\bz)+\|\bm z\|-R.
\]
Hence the radial term is strictly active, and
$\partial h_{\mathbf U}(\bm z)
=\{\tfrac{\bm z}{\|\bm z\|}\}$.
This proves the claimed oracle formula.
If $\|\bm z\|\ge4R$, then
$\|\prox_{\tfrac12 h_{\mathbf U}}(\bm z)\|\ge\|\bm z\|-\tfrac12>3R$
by the triangle inequality.
The radial term is therefore also strictly active
at the proximal point, and \eqref{eq:prox} gives
$\|\nabla(h_{\mathbf U})_{\tfrac12}(\bm z)\|=1$.
 
\textrm{(ii)}
Since at least $N$ coordinates of
$\tfrac23\mathbf U^\top\bm z$ do not exceed $\eta$,
Lemma~\ref{lem:away-from-Z}(i) gives
$\operatorname{dist}
 (\tfrac23\mathbf U^\top\bm z,\mathcal Z)
 \ge\tfrac5{\sqrt N}$.
Suppose, to the contrary, that
$\|\nabla(h_{\mathbf U})_{\tfrac12}(\bm z)\|
\le\tfrac1{32\sqrt N}$.
Then, by \eqref{eq:prox},
\[
\begin{aligned}
 \operatorname{dist}\left(
 \frac23\mathbf U^\top
 \prox_{\tfrac12 h_{\mathbf U}}(\bm z),\mathcal Z
 \right)
 &\ge
 \frac5{\sqrt N}
 -\frac23
 \left\|\bm z-\prox_{\tfrac12 h_{\mathbf U}}(\bm z)\right\|\\
 &=
 \frac5{\sqrt N}
 -\frac13
 \left\|\nabla(h_{\mathbf U})_{\tfrac12}(\bm z)\right\|\ge
 \frac5{\sqrt N}-\frac1{96\sqrt N}>r_0.
\end{aligned}
\]
 
We first consider the subgradients of the first function in
\eqref{eq:radial-instance}
at $\prox_{\tfrac12 h_{\mathbf U}}(\bm z)$.
By the chain rule, each such subgradient has the form
\[
 \bar{\bm g}=\frac13\mathbf U\bm g,
 \qquad
 \bm g\in\partial\bar f\left(
 \frac23\mathbf U^\top
 \prox_{\tfrac12 h_{\mathbf U}}(\bm z)
 \right).
\]
By Lemma~\ref{lem:away-from-Z}(iii) and
$\mathbf U^\top\mathbf U=\mathbf I_K$,
every such $\bar{\bm g}$ satisfies
$
 \|\bar{\bm g}\|
 =\tfrac13\|\bm g\|
 \ge\tfrac1{3\sqrt{98N}}.
$
If only the first function were active at the proximal
point, \eqref{eq:prox} would therefore give
$\|\nabla(h_{\mathbf U})_{\tfrac12}(\bm z)\|
 \ge\tfrac1{3\sqrt{98N}}
 >\tfrac1{32\sqrt N}$,
contradicting our assumption.
Thus the radial term is active at the proximal point.
Its value there is at least $\tfrac12\inf\bar f$, so
\textnormal{(C4)} gives
\begin{equation}
\label{eq:proxbound}
 \left\|\prox_{\tfrac12 h_{\mathbf U}}(\bm z)\right\|
 \ge R-\frac12\bigl(\bar f(\bz)-\inf\bar f\bigr)
 \ge R-\frac{5M}{2N}>0.
\end{equation}
 
We next show that every convex combination of the
subgradients of the two functions has norm greater than
$\tfrac1{32\sqrt N}$.
For every $\bar{\bm g}$ as above,
\eqref{eq:radial-inner-product} gives
\[
 \left\langle
 \bar{\bm g},\prox_{\tfrac12 h_{\mathbf U}}(\bm z)
 \right\rangle
 =\frac12\left\langle
 \bm g,\frac23\mathbf U^\top
 \prox_{\tfrac12 h_{\mathbf U}}(\bm z)
 \right\rangle
 \ge-\frac{9\sqrt K+3M}{N}.
\]
Together with \eqref{eq:proxbound} and $R=2^{18}(M+\sqrt K)$,
we obtain
\[
 \left\langle
 \bar{\bm g},
 \frac{\prox_{\tfrac12 h_{\mathbf U}}(\bm z)}
 {\|\prox_{\tfrac12 h_{\mathbf U}}(\bm z)\|}
 \right\rangle
 \ge-\frac{9\sqrt K+3M}
 {N\left(R-\tfrac{5M}{2N}\right)}
 >-\frac1{28224N}.
\]
By the subdifferential maximum rule and \eqref{eq:prox}, there exist
$\lambda\in[0,1]$ and $\bar{\bm g}$ as above such that
\[
 \nabla(h_{\mathbf U})_{\tfrac12}(\bm z)
 =
 \lambda\bar{\bm g}
 +(1-\lambda)
 \frac{\prox_{\tfrac12 h_{\mathbf U}}(\bm z)}
 {\|\prox_{\tfrac12 h_{\mathbf U}}(\bm z)\|}.
\]
Expanding the squared norm and using the two bounds on
$\bar{\bm g}$ yields
\[
\begin{aligned}
 \left\|\nabla(h_{\mathbf U})_{\tfrac12}(\bm z)\right\|^2
 &\ge
 \frac{\lambda^2}{882N}
 +(1-\lambda)^2
 -\frac{\lambda(1-\lambda)}{14112N}\\
 &\ge
 \frac1{882N+1}-\frac1{56448N}>\frac1{1024N},
\end{aligned}
\]
where the second inequality follows from the Cauchy--Schwarz inequality and $\lambda(1-\lambda)\le\tfrac14$.
This contradiction proves (ii).
\end{proof}

Finally, we verify the counterpart of
\textnormal{(C2)} for $h_{\mathbf U}$.
For every integer $t\ge0$, define
\[
 h_{\mathbf U,t}(\bm z)
 :=
 \max\left\{
 \frac12\bar f\left(\frac23\mathbf\Pi_t\mathbf U^\top\bm z\right),
 \frac12\bar f(\bz)+\|\bm z\|-R
 \right\}.
\]
If $|(\tfrac23\mathbf U^\top\bm z)_{mn}|<\eta$
for all $(m,n)\notin\cI_{t-1}$, then
\textnormal{(C2)} and the chain rule show that the first functions
in the definitions of $h_{\mathbf U}$ and $h_{\mathbf U,t}$ have the
same value and the same subdifferential at $\bm z$.
Since the radial terms are identical, the max rule gives
\begin{equation}
 \mathbb O_{h_{\mathbf U}}(\bm z)
 =
 \mathbb O_{h_{\mathbf U,t}}(\bm z).
 \label{eq:radial-oracle-locality}
\end{equation}
For $\|\bm z\|\ge4R$, the radial terms are strictly
active in both functions by the estimate in
Lemma~\ref{lem:radial-stationarity}(i), so
\eqref{eq:radial-oracle-locality} also holds.
 
\subsection{The random rotation argument}
\label{subsec:random-rotation}
We give a simplified adaptation of the random rotation
arguments of \citet{woodworth2016tight}
and \citet[Section~5.1]{carmon2020lower}
to the full subdifferential oracle.
The proof uses an auxiliary run whose queries depend
only on previously revealed columns.
 
We show that, with high probability, every query and the final
output $\bm z$ with $\|\bm z\|<4R$ satisfy the assumption of
Lemma~\ref{lem:radial-stationarity}(ii), namely that the coordinates
of $\tfrac23\mathbf U^\top\bm z$ in block $M$ are smaller than $\eta$.
Points with $\|\bm z\|\ge4R$ are covered by
Lemma~\ref{lem:radial-stationarity}(i).
\begin{lemma}[Randomized reduction]
\label{lem:randomized-reduction}
Fix $\delta\in(0,1)$ and an integer $d$ satisfying
\begin{equation}
 d\ge K+1+\frac{128R^2}{9\eta^2}
       \log\frac{2K(T+1)}{\delta}.
 \label{eq:rotation-dimension}
\end{equation}
   Let $\mathbf U\in\R^{d\times K}$ be uniformly distributed on the set
   of matrices with orthonormal columns,
   $\{\mathbf U\mid\mathbf U^\top\mathbf U=\mathbf I_K\}$,
   independently of the algorithm's private seed.
For any randomized algorithm accessing
$\mathbb O_{h_{\mathbf U}}$, let
$\bm z^{(0)},\ldots,\bm z^{(T-1)}$ denote its
first $T$ queries, and let $\bm z^{(T)}$ be any
Borel measurable output determined by its private
seed and these $T$ oracle replies.
Then
\[
 \mathbb P\left(
 \min_{0\le t\le T}
 \left\|\nabla(h_{\mathbf U})_{\tfrac12}
                  (\bm z^{(t)})\right\|
 >\frac1{32\sqrt N}
 \right)\ge1-\delta,
\]
where the probability is over $\mathbf U$ and the
algorithm's private seed.
\end{lemma}
 
\begin{proof}
For $(m,n)\in[M]\times[N]$, let
$\bm b_{mn}:=\mathbf U\bm e_{mn}$.
Run the algorithm on $\mathbb O_{h_{\mathbf U}}$,
and denote its queries and output by
$\bm z^{(0)},\ldots,\bm z^{(T)}$.
Using the same seed, consider an auxiliary run
in which query $t<T$ is answered by
$\mathbb O_{h_{\mathbf U,t}}$.
Denote its queries and output by
$\widetilde{\bm z}^{(0)},\ldots,
\widetilde{\bm z}^{(T)}$.
 
The point $\widetilde{\bm z}^{(t)}$ is determined
by the seed and the replies at times $s<t$.
Each such reply uses $h_{\mathbf U,s}$, which
depends on $\mathbf U$ only through the columns
indexed by $\cI_s\subseteq\cI_{t-1}$.
By induction, using the Borel measurability of
the oracle and decision maps, each
$\widetilde{\bm z}^{(t)}$ is a Borel function of
the seed and the columns indexed by $\cI_{t-1}$.
This holds also for $t=T$, since the output uses
only the seed and the preceding $T$ replies.\footnote{The map
$(\bm x,(\bm b_{mn})_{(m,n)\in\cI_t})\mapsto\partial h_{\mathbf U,t}(\bm x)$
has a closed graph and is locally bounded. It is therefore upper
semicontinuous with nonempty compact values, and hence Borel into
$\mathcal K_d$ equipped with the Hausdorff metric; see \citet[Ch.~18]{aliprantis2006infinite}.}

Fix $0\le t\le T$ and $(m,n)\notin\cI_{t-1}$.
Conditional on the seed and the columns indexed
by $\cI_{t-1}$, the point
$\widetilde{\bm z}^{(t)}$ is fixed.
By rotational invariance, $\bm b_{mn}$ is uniform
on the unit sphere in the orthogonal complement of these columns,
whose dimension is $d-|\cI_{t-1}|\ge d-K$.
Projecting $\widetilde{\bm z}^{(t)}$ onto this
complement preserves its inner product with
$\bm b_{mn}$ and does not increase its norm.
Hence, if $\|\widetilde{\bm z}^{(t)}\|<4R$ and
$\left|\tfrac23\langle\bm b_{mn},\widetilde{\bm z}^{(t)}\rangle\right|\ge\eta$,
then the inner product of $\bm b_{mn}$ with the normalized projection
has absolute value at least $\tfrac{3\eta}{8R}$.
By the spherical cap bound~\citep[Lemma~2.2]{ball1997elementary}
and symmetry, this event has conditional probability at most
$2\exp\bigl(-\tfrac{9(d-K)\eta^2}{128R^2}\bigr)$.
Averaging over the seed and the conditioned columns
therefore yields
\[
 \mathbb P\left(
 \|\widetilde{\bm z}^{(t)}\|<4R,\ 
 \left|\frac23
 \langle\bm b_{mn},\widetilde{\bm z}^{(t)}\rangle
 \right|\ge\eta
 \right)
 \le
 2\exp\left(
 -\frac{9(d-K)\eta^2}{128R^2}
 \right).
\]
A union bound over at most $K(T+1)$ pairs of
times and coordinates, together with
\eqref{eq:rotation-dimension}, gives
$\mathbb P(\mathcal E)\ge1-\delta$,
where $\mathcal E$ is the event that, simultaneously
for every $0\le t\le T$,
\[
 \|\widetilde{\bm z}^{(t)}\|<4R
 \quad\Longrightarrow\quad
 \left|\frac23
 \langle\bm b_{mn},\widetilde{\bm z}^{(t)}\rangle
 \right|<\eta
 \quad\text{for all }(m,n)\notin\cI_{t-1}.
\]
On $\mathcal E$, the two oracles return the same reply
at every auxiliary query by
\eqref{eq:radial-oracle-locality}: $\mathbb O_{h_{\mathbf U}}(\widetilde{\bm z}^{(t)})
 =
 \mathbb O_{h_{\mathbf U,t}}(\widetilde{\bm z}^{(t)}),$ for all $0\le t<T$. 
Both runs start at $\bz$ and use the same seed.
Whenever their preceding replies agree, they produce
the same next query, and the preceding identity gives
the same reply.
Induction therefore yields identical transcripts
and outputs on $\mathcal E$, so
$\bm z^{(t)}=\widetilde{\bm z}^{(t)}$
for all $0\le t\le T$.

Still on $\mathcal E$, fix $0\le t\le T$.
If $\|\bm z^{(t)}\|\ge4R$, then
${\|\nabla(h_{\mathbf U})_{\tfrac12}(\bm z^{(t)})\|=1>\tfrac1{32\sqrt N}}$
by Lemma~\ref{lem:radial-stationarity}(i).
Otherwise, since
\[
 \operatorname{lev}(M,n)=(M-1)\ell+n-1=T+n-1\ge T>t-1,
 \qquad n\in[N],
\]
all coordinates of block $M$ lie outside $\cI_{t-1}$.
The definition of $\mathcal E$ and
$\bm z^{(t)}=\widetilde{\bm z}^{(t)}$ therefore give
\[
 \left|
 \left(\frac23\mathbf U^\top\bm z^{(t)}\right)_{Mn}
 \right|<\eta,
 \qquad n\in[N].
\]
Thus at least $N$ coordinates of $\tfrac23\mathbf U^\top\bm z^{(t)}$
do not exceed $\eta$, and
Lemma~\ref{lem:radial-stationarity}(ii) gives
$\|\nabla(h_{\mathbf U})_{\tfrac12}(\bm z^{(t)})\|>\tfrac1{32\sqrt N}$.
Since $0\le t\le T$ was arbitrary and
$\mathbb P(\mathcal E)\ge1-\delta$,
this proves the lemma.
\end{proof}
\begin{remark}[Dimension of the hard instances]
    For deterministic algorithms, a finite-horizon resisting oracle argument
\citep[Section~3.3]{carmon2020lower} makes the unrevealed
directions orthogonal to the queries and final output,
so dimension $d=O(T)$ suffices. For randomized algorithms, the hard instance must be
chosen independently of the realized random seed.
Our reduction uses concentration to keep the unrevealed
coordinates below $\eta$ for queries of norm at most $4R$. For fixed $G,\rho,\Delta$ and $\delta$, the scaling below
gives $R=\Theta(\sqrt T)$ and $\eta=\Theta(T^{-3/2})$,
so \eqref{eq:rotation-dimension} yields
$
 d=O\!\left((\tfrac{R}{\eta})^2\log T\right)
  =O(T^4\log T).
$ Thus the dimension requirement reflects the radius and
robustness tolerance of our construction.
\end{remark}
\subsection{Scaling} 
Choose $ 
 N:=\lfloor\tfrac{G^2}{1024\epsilon^2}\rfloor$ and $
 M:=\lfloor\tfrac{\rho\Delta N}{3G^2}\rfloor$.
The assumption
$\epsilon\le\tfrac1{512}\min\{G,\sqrt{\rho\Delta}\}$
gives $N\ge256$ and $M\ge2$.
Moreover, the upper bound on $\Delta$ and the definition
of $N$ imply
\[
 M\le\frac N3\exp\left(\frac{G^2}{200\epsilon^2}\right)
 <\frac N3\exp\left(\frac{128}{25}(N+1)\right)
 \le2^{8N},
\]
where the last inequality holds for $N\ge256$.
Thus $(M,N)$ is valid.
 
Since $N\ge \tfrac{G^2}{2048\epsilon^2}$ and
$M\ge\tfrac{\rho \Delta N}{6G^2}$, we obtain
\[
 T\ge\frac{MN}{16}
 \ge\frac{\rho\Delta N^2}{96G^2}
 \ge2^{-29}\frac{\rho G^2\Delta}{\epsilon^4}.
\]
Define
\[
 f_{\mathbf U}(\bm z)
 :=\frac{G^2}{\rho}
 h_{\mathbf U}\left(\frac\rho G\bm z\right).
\] 
Then $f_{\mathbf U}$ is $G$-Lipschitz and
$\rho$-weakly convex, and
\[
 f_{\mathbf U}(\bz)-\inf f_{\mathbf U}
 \le\frac{3G^2M}{\rho N}\le\Delta,
 \qquad
 \nabla(f_{\mathbf U})_{\tfrac1{2\rho}}(\bm z)
 =G\nabla(h_{\mathbf U})_{\tfrac12}
       \left(\frac\rho G\bm z\right).
\]
An oracle query to $f_{\mathbf U}$ at $\bm z$
is answered by querying $h_{\mathbf U}$ at
$\tfrac\rho G\bm z$ and scaling the returned
value and subdifferential by $\tfrac{G^2}{\rho}$ and $G$,
respectively. 
 
Fix any randomized algorithm $\sA$.
Running $\sA$ on $f_{\mathbf U}$ is thus equivalent to running, with
the same seed, an algorithm on $h_{\mathbf U}$ whose $t$-th query is
$\tfrac\rho G\sA^{(t)}[f_{\mathbf U}]$.
Set $\delta:=\tfrac{1-\nu}{2}$ and choose $d$ satisfying
\eqref{eq:rotation-dimension}.
Apply Lemma~\ref{lem:randomized-reduction} to this algorithm, with
$\tfrac\rho G\sA^{(T)}[f_{\mathbf U}]$ as its output. Since $\tfrac{G}{32\sqrt N}\ge\epsilon$ and $ 1-\delta=\tfrac{1+\nu}{2}>\nu$, 
averaging over $\mathbf U$ yields a fixed matrix
$\mathbf U$ such that
\[
 \mathbb P_{\xi_{\sA}}\left(
 \min_{0\le t\le T}
 \left\|
 \nabla(f_{\mathbf U})_{\tfrac1{2\rho}}
 \bigl(\sA^{(t)}[f_{\mathbf U}]\bigr)
 \right\|>\epsilon
 \right)>\nu.
\]
Thus, on this fixed instance, no candidate through
time $T$ succeeds with probability at least $1-\nu$.
Since $\sA$ was arbitrary,
$\cT_{\epsilon,\nu}\ge T+1$, proving
Theorem~\ref{thm:main}.

\section*{AI Disclosure}
The hard instance construction was developed through iterative collaboration with GPT-5.6 Sol. The authors developed the conceptual framework, guided the construction, independently verified the resulting arguments, and rewrote the manuscript. GPT-6 Astra subsequently helped improve the exposition. The accompanying Lean~4 formalization, developed with Codex, verifies
Theorem~\ref{thm:main} with the stated constants and is available at
\url{https://github.com/SiyuPan04/Weakly-Convex-Lower-Bound-Lean}. The authors take full responsibility for all content.

\section*{Acknowledgements}
 Jiajin Li  was supported by a Natural Sciences and Engineering Research Council of Canada
Discovery Grant RGPIN-2025-05817. Jiajin Li thanks Lai Tian for helpful discussions.
 
\bibliographystyle{abbrvnat}
\bibliography{ref}
 
\end{document}